\documentclass{amsart}
\usepackage[utf8]{inputenc}
\usepackage[english]{babel}
\usepackage{enumerate}
\usepackage{algorithm}
\usepackage{algorithmicx}
\usepackage{algpseudocode}
\usepackage{graphicx}
\usepackage{amssymb}
\usepackage{amsmath}
\usepackage{amsthm}
\usepackage{leftidx}
\usepackage{mathdots}
\usepackage{amssymb}
\usepackage{mathrsfs}
\usepackage{graphicx}
\usepackage{epsfig}
\usepackage{bm}
 \usepackage{color}
 \usepackage{subfig}
\usepackage{multirow}
\usepackage{algorithm}
\usepackage{algorithmicx}
\usepackage{algpseudocode}
\usepackage{graphicx}
\usepackage{calc}
\usepackage{epsfig}
\usepackage{verbatim}
 \numberwithin{equation}{section}
\newcommand{\R}{{\mathbb R}}
\newcommand{\Z}{{\mathbb Z}}

\newcommand{\bw}{{\mathbf w}}
\newcommand{\C}{{\mathbb C}}

\newcommand{\Tr}{{\rm Tr}}

\renewcommand{\eqref}[1]{(\ref{#1})}

\newcommand{\innerp}[1]{\langle #1 \rangle}

\newcommand{\abs}[1]{\lvert#1\rvert}

\newcommand{\A}{{\mathcal A}}

\newcommand{\rank}{{\rm rank}}

\newcommand{\vx}{{\mathbf x}}

\newcommand{\vy}{{\mathbf y}}
\newcommand{\vz}{{\mathbf z}}
\newcommand{\vw}{{\mathbf w}}

\newcommand{\va}{{\mathbf a}}
\newcommand{\vb}{{\mathbf b}}

\newcommand{\vt}{{\mathbf t}}
\newcommand{\vu}{{\mathbf u}}

\renewcommand{\H}{{\mathbb F}}
\newcommand{\HF}{{\H}}

\newcommand{\FF}{{\mathbb F}}

\newcommand{\G}{{\mathcal G}}

\newcommand{\argmin}[1]{\mathop{\rm argmin}\limits_{#1}}
\newtheorem{prop}{Proposition}[section]
\newtheorem{lem}[prop]{Lemma}
\newtheorem{defi}{Definition}[section]
\newtheorem{coro}[prop]{Corollary}
\newtheorem{theo}[prop]{Theorem}
\newtheorem{remark}[prop]{Remark}

\newtheorem{exam}{Example}[section]

\theoremstyle{plain}

\ifx\proof\undefined
\newenvironment{proof}[1][\protect\proofname]{\par
\normalfont\mathrm{T}sep6\p@\@plus6\p@\relax
\trivlist
\itemindent\parindent
\item[\hskip\labelsep\scshape #1]\ignorespaces
}{%
\endtrivlist\@endpefalse
}
\providecommand{\proofname}{Proof}
\fi

\usepackage{babel}

\makeatother

\usepackage{babel}

\providecommand{\theoremname}{Theorem}
\begin{document}

\title{Sparse phase retrieval via Phaseliftoff  }

\author{Yu Xia}
\thanks{Yu Xia was supported by NSFC grant (11901143), Zhejiang Provincial Natural Science Foundation (LQ19A010008), Education Department of Zhejiang Province Science Foundation (Y201840082). }
\address{Department of Mathematics,
Hangzhou Normal University, Hangzhou 311121, China } \email{yxia@hznu.edu.cn}

\author{Zhiqiang Xu}
\thanks{Zhiqiang Xu was supported  by Beijing Natural Science Foundation (Z180002)  and by
 NSFC grant (11688101)}
\address{LSEC, Inst.~Comp.~Math., Academy of
Mathematics and System Science,  Chinese Academy of Sciences, Beijing, 100091, China
\newline
School of Mathematical Sciences, University of Chinese Academy of Sciences, Beijing
100049, China} \email{xuzq@lsec.cc.ac.cn}


\begin{abstract}
The aim of sparse phase retrieval  is to recover a $k$-sparse signal $\mathbf{x}_0\in
\mathbb{C}^{d}$ from  quadratic measurements $|\langle
\mathbf{a}_i,\mathbf{x}_0\rangle|^2$ where $\va_i\in \C^d, i=1,\ldots,m$. Noting
$|\langle \mathbf{a}_i,\mathbf{x}_0\rangle|^2={\rm Tr}(A_iX_0)$ with
$A_i=\va_i\va_i^*\in \C^{d\times d}, X_0=\vx_0\vx_0^*\in \C^{d\times d}$, one can
recast sparse  phase retrieval as a problem of recovering  a  rank-one sparse matrix
from linear measurements. Yin and Xin  introduced {\em
PhaseLiftOff} which presents a proxy of rank-one condition via the difference of
trace and Frobenius norm. By adding sparsity penalty to {\em PhaseLiftOff}, in this
paper, we present a novel model to recover sparse signals from quadratic
measurements. Theoretical analysis shows that the solution to our model provides the
stable recovery of  $\mathbf{x}_0$ under almost optimal sampling complexity
$m=O(k\log(d/k))$. The computation of our model is carried out by the difference of
convex function algorithm (DCA). Numerical experiments demonstrate that our algorithm
outperforms other state-of-the-art algorithms used for solving sparse phase
retrieval.
\end{abstract}

	\maketitle
\bigskip \medskip

\section{Introduction}
\subsection{Phase retrieval}
We  assume that $\vx_0\in \FF^d$ is a target signal, where $\FF\in \{\R,\C\}$. The
aim of phase retrieval is to recover $\vx_0\in {\FF}^d$ from
$\abs{\innerp{\va_j,\vx_0}}^2+w_j, j=1,\ldots,m$,  up to a unimodular constant where
$\va_j\in \FF^d$ are known measurement vectors and $\bw:=(w_1,\ldots,w_m)^T\in \R^m$
is a noise vector. Phase retrieval is raised in many areas, such as X-ray
crystallography, astronomy, quantum tomography, optics and microscopy. For
convenience,  let $\A: \HF^{d\times d}\rightarrow \R^m$ be a linear map which is
defined as
\begin{equation}\label{eqn: A()}
\A(X)=(\va_1^*X\va_1,\ldots,\va_m^*X\va_m),
\end{equation}
where $X\in \HF^{d\times d}, \va_j\in \FF^d, j=1,\ldots,m$. We abuse the notation and
set
\[
\A(\vx):=\A(\vx\vx^*)=(\abs{\innerp{\va_1,\vx}}^2,\ldots,\abs{\innerp{\va_m,\vx}}^2),
\]
where $\vx\in \HF^d$.  With these notations, we can formulate the aim of  phase
retrieval as follows:  {\em To estimate the matrix $X_0=\vx_0\vx_0^*\in \C^{d\times
d}$ from $\A(\vx_0)+\bw\in \R^m$}.

For the noiseless  case, to guarantee the solution $\A(\vx)=\A(\vx_0)$ is unique for
all $\vx_0\in \C^d$, it is shown in \cite{WX19} that the measurement number $m\geq
4d-2-2\alpha_d$ is necessary where $\alpha_d$ denotes the number of $1$'s in the
binary of expansion of $d-1$. The authors in \cite{algphas} proved that $m\geq 4d-4$
generic measurement vectors $\va_j\in \C^d, j=1,\ldots,m,$ are enough to guarantee
the uniqueness of the solution.

In \cite{phaselift1,phaselift2,phaselift3}, the phase retrieval was recasted as a
semi-definite programming problem, i.e., the PhaseLift problem:
\begin{equation}\label{eq:phaselift}
\min_{X\in{{\mathbb{F}}}^{d\times d}} {\rm Tr}(X) \quad {\rm s.t.} \quad \A(X)=\A(X_0),\  X\succeq 0.
\end{equation}
In \cite{phaselift2}, it is shown that the solution to (\ref{eq:phaselift}) is $X_0$
with high probability provided $\va_j$ is Gaussian random vector and $m=O(d\log d)$,
which was reduced  to $m=O(d)$ in \cite{phaselift1}. For the aim of computation, the
regularized trace-norm minimization is suggested in \cite{phaselift2,phaselift3}:
\begin{equation}\label{eq:regul}
\min_{X\succeq 0, X\in {{\mathbb{F}}}^{d\times d}} \frac{1}{2}\|\A(X)-\vb\|_2^2 +\lambda {\rm Tr}(X).
\end{equation}

Noting that $\Tr(X)-\|X\|_F\geq 0$ and the equality holds iff ${\rm rank}(X)=1$, Yin
and Xin suggested the following regularization problem \cite{phaseliftoff}, which is
called as  {\em PhaseLiftOff}:
\begin{equation}\label{eq:phaseliftoff}
\min_{X\succeq 0, X\in {{\mathbb{F}}}^{d\times d}} \frac{1}{2}\|\A(X)-\vb\|_2^2 +\lambda (\Tr(X)-\|X\|_F).
\end{equation}
The numerical experiments in \cite{phaseliftoff} showed that PhaseLiftOff outperforms
PhaseLift.
\subsection{Sparse phase retrieval}
In many areas, one also requires $\|\vx_0\|_0\leq k$, i.e., the number of nonzero
entries of $\vx_0$ less than or equal to $k$ \cite{sparse1,sparse2,sparse3}. The aim
of {\em sparse phase retrieval} is to recover the $k$-sparse signal $\vx_0$ from
$\abs{\innerp{\va_j,\vx_0}}^2=b_j, j=1,\ldots, m$.

For convenience, we set $\Sigma_{k}^\FF:=\{\vx\in \FF^d: \|\vx\|_0\leq k\}$. It was
shown in \cite{sparse2} that, for $\FF=\C$ and $\vx_0\in \Sigma_k^\FF$, if $m\geq
4k-2$ (resp. $m\geq 2k$ for $\FF=\R$) and $\va_1,\ldots,\va_m$ are generic vectors in
$\C^d$ (resp. $\R^d$) then the solution to $\A(\vx)=\A(\vx_0)$ with $\vx\in \Sigma_k^\FF$
is unique up to a unimodular constant.

The $\ell_1$-minimization is a commonly used method for recovering sparse signals.
Naturally, one is also interested in employing $\ell_1$-minimization for solving
sparse phase retrieval. For $\FF=\R$, the following model was considered in
\cite{VX}:
\begin{equation}\label{eq:L1}
\min_{\vx\in \mathbb{R}^{d\times d}} \|\vx\|_1 \quad {\rm s.t.} \quad \A(\vx)=\A(\vx_0).
\end{equation}
Particularly, it is proved that the solution to (\ref{eq:L1}) is $\pm\vx_0$ with high
probability if $m\gtrsim k\log d$ and $\va_j, j=1,\ldots,m,$ are independent Gaussian
random vectors. In \cite{XiaXu}, the authors extended this result  to the case where
$\FF=\C$.

 In \cite{LV}, the following convex model was considered
\begin{equation}\label{eq:LV}
\min_{X\in \R^{d\times d}} \|X\|_1+\lambda \text{Tr}(X), \quad {\rm s.t.} \quad
\A(X)=\A(\vx_0), \ X \succeq 0.
\end{equation}
The objective function in (\ref{eq:LV}) is the summation of the trace and the $\ell_1$
norm, which is also a convex model.
 To guarantee the solution  to (\ref{eq:LV}) is $\vx_0\vx_0^*$,  one has to require the
number of measurements $m\gtrsim k^2 \log d$, which is quadratic about the sparse
level $k$ \cite{LV}.

Beyond the convex model, one also develops  many nonconvex algorithms for solving
 sparse phase retrieval, such as Sparse Truncated Amplitude flow (SPARTA) \cite{WZG16},
  Thresholded Wirtinger Flow (ThWF) \cite{CLM}, Sparse Wirtinger Flow (SWF)
  \cite{YWW17},
   Sparse Phase
Retrieval via Smoothing Function (SPRSF) \cite{PBA18}.  These algorithms include two stages: (i)
Recover the support of the underlying sparse signal under some analytical rule, and
construct an initialization near the ground truth signal $\mathbf{x}_0$; (ii) Refine
the initialization by gradient-type iterations and extra truncation procedure by hard
thresholding. However, to guarantee the algorithms converge to the true signal, the
algorithms mentioned above require the sample complexity is $m=O(k^2\log d)$.


\subsection{Our contribution}
A natural model for sparse phase retrieval is to use $\ell_1$-regularization methods,
i.e.,
\begin{equation}\label{eqn: unconstrained_lowrank}
\min_{X\in {{\mathbb{C}}}^{d\times d}} \mu \|X\|_1+\frac{1}{2}\|\mathcal{A}(X)-\vb\|_2^2,\quad {\rm s.t.}\quad X\succeq 0, \ {\rank}(X)=1,
\end{equation}
where $\|X\|_1=\sum_{l=1}^d\sum_{j=1}^d\abs{x_{j,l}}$ and $x_{j,l}$ are the entries
of $X$.
 Motivated by the notable PhaseLiftOff \cite{phaseliftoff}, we reformulate
(\ref{eqn: unconstrained_lowrank}) as the following regularization problem:
\begin{equation}\label{eq: unconstrained SDP}
\min_{X\in {{\mathbb{C}}}^{d\times d}} \lambda (\text{Tr}(X)-\|X\|_F)+\mu \|X\|_1+\frac{1}{2}\|\mathcal{A}(X)-\vb\|_2^2
\quad {\rm s.t.}\quad X\succeq 0.
\end{equation}
For convenience, we call (\ref{eq: unconstrained SDP}) {\em Sparse PhaseLiftOff}
model.
 Note that the object function in (\ref{eq: unconstrained SDP}) is
the difference of convex functions and hence it can be solved by the {\em difference
of convex functions algorithm} (DCA).

 To study the performance of (\ref{eq: unconstrained SDP}), we first establish
the equivalence between (\ref{eq: unconstrained SDP}) and (\ref{eqn:
unconstrained_lowrank}) under some mild conditions about $\lambda$, $\mu$ and
$\|\vw\|_2$:
\begin{lem}\label{le:xia}
Assume that $\vb=\mathcal{A}(\vx_0\vx_0^*)+\mathbf{\vw}$  where $\vx_0\in \C^d$, and
$\vw\in \R^m$ is the
 noise term. Let $X^\#$ be the global minimizer of
(\ref{eq: unconstrained SDP}).  If $\frac{1}{2}\|\vb\|_2^2> \mu
\|\vx_0\|_1^2+\frac{1}{2}\|\vw\|_2^2$, $\mu\geq 0$ and
\begin{equation}\label{eq:lam}
\lambda >\frac{ \mu d+\|\mathcal{A}\|(\sqrt{2\mu}\|\vx_0\|_1+\|\vw\|_2)}{\sqrt{2}-1},
\end{equation}
then   $\text{rank}(X^\#)=1$.
\end{lem}
As said before,  ${\rm Tr}(X)-\|X\|_F=0$ provided ${\rm rank}(X)=1$. Under the
conditions of Lemma \ref{le:xia},  $X^\#$ is also the minimizer of (\ref{eqn:
unconstrained_lowrank}).
 Hence, we turn to study the performance of
(\ref{eqn: unconstrained_lowrank}). To do that, we require $\A$ satisfies restricted
isometry property over low-rank and sparse matrices:
\begin{defi}\cite{XiaXu}
We say that the  map $\A:{\mathbb H}^{d\times d}\rightarrow \R^m$ satisfies the
restricted isometry property of order $(r,s)$ if there exist positive constants $c$
and $C$ such that the inequality
\begin{equation}\label{de:RIP}
c\|X\|_{F}\leq\frac{1}{m}\|\mathcal{A}(X)\|_{1}\leq C\|X\|_{F}
\end{equation}
holds for all $X\in \mathbb{H}^{d\times d}$ with $\text{rank}(X)\leq r$ and
$\|X\|_{0,2}\leq s$.
\end{defi}
Then, we
have the following theorem.
\begin{theo}\label{th:rank1mod}
Let $\vb=\mathcal{A}(\vx_0\vx_0^*)+\mathbf{\vw}$,  where $\vx_0\in \C^d$, and
$\vw\in \R^m$ is the
 noise term.  Assume that $\mathcal{A}(\cdot)$ satisfy the RIP condition of order $(2,2ak)$ with
RIP constant $c, C>0$, and $a\geq 1$ with
\begin{equation}\label{eqn: condition_a}
c-\frac{4\sqrt{2}C}{\sqrt{a}}-\frac{C}{a}>0.
\end{equation}
Set
$
\alpha:=\frac{1+\frac{1}{a}+\frac{4\sqrt{2}}{\sqrt{a}}}{c-\frac{C}{a}-\frac{4\sqrt{2}C}{\sqrt{a}}}.
$
Assume that $\mu >0$. For any $k$-sparse signals $\vx_0\in \C^d$,  the solution to
(\ref{eqn: unconstrained_lowrank}) $X^\#:=\vx^\#(\vx^\#)^* $ satisfies
\begin{equation}\label{eqn: conclusion1}
\|\vx^\#(\vx^\#)^* -\vx_0\vx_0^*\|_F\leq
(2C\alpha+2)\frac{\|\vw\|_2}{\sqrt{\mu ak}}\left(\|\vx_0\|_1+\frac{\|\vw\|_2}{\sqrt{\mu}}\right)+\frac{\|\vw\|_2^2}{{2\mu ak}}+\alpha\cdot\frac{\mu ak}{2C}\left(\frac{C\|\vw\|_2 }{\mu ak}+\frac{1}{\sqrt{m}}\right)^2.
\end{equation}

\end{theo}

Combing Lemma \ref{le:xia}, Theorem \ref{th:rank1mod} and Theorem \ref{thm: RIP}, we
have the following corollary:
\begin{coro}\label{th:main}
Assume that $\va_j, j=1,\ldots, m$ are independently  complex Gaussian random
vectors, i.e., $\va_j\sim\mathcal{N}(0,
\frac{1}{2}\mathbf{I}_{d})+\mathcal{N}(0,\frac{1}{2}\mathbf{I}_{d})i$. Assume that
$\vb=\mathcal{A}(\vx_0\vx_0^*)+\vw$  where  $\vw\in \mathbb{R}^m$ is a noise vector
and $\vx_0\in \C^d$ with $\|\vx_0\|_0\leq k$. Assume that $m\gtrsim k\log(ed/k)$. Let
$X^\#$ be the global minimizer of model (\ref{eq: unconstrained SDP}). The following
holds  with probability at least $1-\exp(-cm)$: If $\frac{1}{2}\|\vb\|_2^2> \mu
\|\vx_0\|_1^2+\frac{1}{2}\|\vw\|_2^2$, $\mu >0$ and
\[
\lambda >\frac{ \mu d+\|\mathcal{A}\|(\sqrt{2\mu}\|\vx_0\|_1+\|\vw\|_2)}{\sqrt{2}-1},
\]
then ${\rm rank}(X^\#)= 1$, and $X^\#=\vx^\#(\vx^\#)^*$ satisfies
\[
\|\vx^\#(\vx^\#)^* -\vx_0\vx_0^*\|_F\lesssim \frac{\|\vw\|_2}{\sqrt{\mu k}}\left(\|\vx_0\|_1+\frac{\|\vw\|_2}{\sqrt{\mu}}\right)+\frac{\|\vw\|_2^2}{{2\mu k}}+{\mu k}\left(\frac{\|\vw\|_2 }{\mu k}+\frac{1}{\sqrt{m}}\right)^2.
\]
\end{coro}
\begin{remark} According to Corollary \ref{th:main}, the parameter $\lambda$ depends
on $\|{\mathcal A}\|$. We next show $\|\mathcal{A}\|=O((m+d)\sqrt{d})$.
 We assume
that   the singular value decomposition of $X\in \C^{d\times d}$ with $\|X\|_F=1$ is
$X=\sum_{j=1}^d\sigma_j\mathbf{u}_j\mathbf{v}_j^*$. Here,
$\sum_{j=1}^d\sigma_{j}^2=1$. We claim that
$\|\mathcal{A}(\mathbf{u}_j\mathbf{v}_j^*)\|_1=O(m+d)$ holds with probability at
least $1-\exp(-m)$.
 Then we have
\[
\|\mathcal{A}(X)\|_2\leq \|\mathcal{A}(X)\|_1
\leq \sum_{j}\sigma_j\|\mathcal{A}(\mathbf{u}_j\mathbf{v}_j^*)\|_1=O((m+d)\sqrt{d}).
\]
 Note that
\[
\|\mathcal{A}(\mathbf{u}_j\mathbf{v}_j^*)\|_1=\sum_{i} |\langle \va_i, \vu_j\rangle
\langle \va_i, \mathbf{v}_j\rangle| \leq \sqrt{\sum_{j}|\langle \va_i,
\vu_j\rangle|^2}\sqrt{\sum_{j}|\langle \va_i, \mathbf{v}_j\rangle|^2}\leq
(\sqrt{m}+\sqrt{d}+t)^2
\]
holds with probability larger than $1-\exp(-t^2/2)$. Here, the last inequality
follows from the singular values of Gaussian random matrices \cite[Corrollary 5.35]{V10}. By taking $t=\sqrt{m}+\sqrt{d}$, we have
$\|\mathcal{A}(\mathbf{u}_j\mathbf{v}_j^*)\|_1=O(m+d)$.
\end{remark}
\begin{remark}
If we take $\|\vw\|_2=0$ in Corollary \ref{th:main}  then the following holds with
high probability:
\[
\|\vx^\#(\vx^\#)^* -\vx_0\vx_0^*\|_2\lesssim \frac{\mu k}{m}
\]
provided
\[
\lambda >\frac{ \mu d+\|\mathcal{A}\|(\sqrt{2\mu}\|\vx_0\|_1)}{\sqrt{2}-1}.
\]
\end{remark}

\subsection{Notations}
We use $\mathbb{H}^{d\times d}$ to denote  the set of all $d\times d$ Hermitian
matrices. For any $X, Y\in \mathbb{H}^{d\times d}$, set $\langle X, Y\rangle
:=\text{Tr}(X^*Y)$. For $x\in \mathbb{C}$, we use $\mathscr{R}(x)$ and
$\mathscr{I}(x)$ to denote the real and complex parts of $x$, respectively. For $X\in
\mathbb{C}^{d\times d}$, we use $X_{i,:}$ and $X_{:,j}$ to denote the $i$-th row and
$j$-th column of $X$, respectively.  For $S, T\subset\{1,\ldots,d\}$, we  use
$X_{S,T}$ to denote a submatrix of $X$ with the rows indexed in $S$ and columns
indexed in $T$.
 We also set
$\|X\|_1:=\sum_{i,j}{\sqrt{\mathscr{R}(X_{i,j})^2+\mathscr{I}(X_{i,j})^2}}$,
$\|X\|_F:=\sqrt{\sum_{i,j}(\mathscr{R}(X_{i,j})^2+\mathscr{I}(X_{i,j})^2)}$, and
$\|X\|_{1,2}:=\sum_{j}\|X_{:,j}\|_2$. We use $\|X\|_{0,2}$ to denote  the number of
non-zero columns in $X$ and use $\text{vec}(X)\in \C^{d^2}$ to denote the
vectorization of $X\in \mathbb{C}^{d\times d}$.
\subsection{Organization}
The paper is organized as follows. After introducing some useful lemmas in Section 2,
we present the proof of Theorem \ref{th:rank1mod} in Section 3. The proof of Lemma
\ref{le:xia} is presented in Section 4. In Section 5, we make a lot of numerical
experiments, which show our method has better performance over the other known
algorithms for sparse phase retrieval.

\section{ Preliminaries and Lemmas}
The following theorem shows that complex Gaussian random quadratic  map $\A$ satisfies RIP of
order $(2,s)$ with high probability provided $m\gtrsim s\log (d/s)$.

\begin{theo}\cite{XiaXu}\label{thm: RIP}
Assume that the linear measurement $\mathcal{A}(\cdot)$ is defined as
\[
\mathcal{A}(X)=(\va_1^*X\va_1,\ldots,\va_m^*X\va_m),
\]
where  $\va_j$ are independently  complex Gaussian random vectors, i.e.,
$\va_j\sim\mathcal{N}(0, \frac{1}{2}\mathbf{I}_{n\times
n})+\mathcal{N}(0,\frac{1}{2}\mathbf{I}_{n\times n})i$. If
\[
 m\gtrsim s\log (d/s),
 \]
 under the probability at least $1-2\exp(-c_0m)$, the linear map $\mathcal{A}$
satisfies the restricted isometry property  of order $(2,s)$, i.e.
\[
0.12\|X\|_{F}\leq\frac{1}{m}\|\mathcal{A}(X)\|_{1}\leq 2.45\|X\|_{F},
\]
for all $X\in \mathbb{H}^{n\times n}$ with $\text{rank}(X)\leq 2$ and
$\|X\|_{0,2}\leq s$ (also $\|X^{*}\|_{0,2}\leq s$).
\end{theo}
We also need the following lemma.
\begin{lem}\cite{XiaXu}\label{lem: MatrixToVector}
If $\vx, \vy\in\mathbb{C}^{d}$, and $\innerp{\vx,\vy} \geq 0$. Then
\[
\|\vx\vx^{*}-\vy\vy^{*}\|_{F}^{2}\geq\frac{1}{2}\|\vx\|_{2}^{2}\|\vx-\vy\|_{2}^{2}.
\]
\end{lem}
The following lemma follows from the proof of Theorem 3.1 in  \cite{phaseliftoff}. We
include a proof here for completeness.
\begin{lem}\label{lem: SDP_property}
Suppose that $\langle X, Y\rangle=0$, where $X, Y\succeq 0$. If $\text{rank}(X)=r\geq 1$, then
\begin{equation}
\left\|\lambda\left(\mathbf{I}-\frac{X}{\|X\|_F}\right)-Y\right\|_F\geq \lambda(\sqrt{r}-1).
\end{equation}
Here $\lambda$ is a non-negative constant.
\end{lem}
\begin{proof}
First of all, we have
\begin{equation}
\label{eqn: low1}
\left\|\lambda\left(\mathbf{I}-\frac{X}{\|X\|_F}\right)-Y\right\|_F\geq \left\|\lambda\mathbf{I}-Y\right\|_F-\lambda\left\|\frac{X}{\|X\|_F}\right\|_F=\left\|\lambda\mathbf{I}-Y\right\|_F-\lambda,
\end{equation}
Then we estimate the lower bound of $\left\|\lambda\mathbf{I}-Y\right\|_F$. Suppose
that  the singular value decomposition of $X$ is in the form of
\[
X=\sum_{i=1}^r\sigma_i\vu_i \vu_i^*,
\]
 where $U_1=(\vu_1,\ldots,\vu_r)\in
\mathbb{C}^{d\times r}$,  and $\sigma_i>0$, for $i=1,\ldots,r$.  Construct  $U_2\in \mathbb{C}^{d\times (d-r)}$,  which satisfies
  $\mathbf{I}=U_1U_1^*+U_2U_2^*$ and $U_1^*U_2=0$. Then we have
   \begin{equation}
 \label{eqn: low20}
 \begin{aligned}
 \left\|\lambda\mathbf{I}-Y\right\|_F^2&=\left\|\lambda(U_1U_1^*+U_2U_2^*)-Y\right\|_F^2=\left\|\lambda U_1U_1^*+(\lambda U_2U_2^*-Y)\right\|_F^2\\
 &=\|\lambda U_1U_1^*\|_F^2+\|\lambda U_2U_2^*-Y\|_F^2+2\lambda \langle U_1U_1^*, \lambda U_2U_2^*-Y\rangle\\
 &=\|\lambda U_1U_1^*\|_F^2+\|\lambda U_2U_2^*-Y\|_F^2-2\lambda \langle U_1U_1^*, Y\rangle.
 \end{aligned}
 \end{equation}
  The last line follows from $U_1^*U_2=0$.  Since $\sigma_i>0$    and  $Y\succeq 0$, the condition $\langle X,Y\rangle=0$ implies that
 \[
 0=\langle X,Y\rangle=\sum_{i=1}^r\sigma_i\langle \vu_i\vu_i^*,Y\rangle=\sum_{i=1}^r\sigma_i\text{Tr}(\vu_i^*Y\vu_i),
 \]
 which leads to $\langle \vu_i\vu_i^*,Y\rangle= 0,$ $ i=1,\ldots,r$. Therefore, it obtain that
\begin{equation}
\label{lem2: temp}
\langle U_1U_1^*,Y\rangle=\sum_{i=1}^r \langle \vu_i\vu_i^*,Y\rangle=0,
\end{equation}
and (\ref{eqn: low20}) becomes
 \begin{equation}
 \label{eqn: low2}
 \left\|\lambda\mathbf{I}-Y\right\|_F^2 = \|\lambda U_1U_1^*\|_F^2+\|\lambda U_2U_2^*-Y\|_F^2\geq  \|\lambda U_1U_1^*\|_F^2=\lambda^2 r.
 \end{equation}
  Combining (\ref{eqn: low1}) and (\ref{eqn: low2}), we have
 \[
 \left\|\lambda\left(\mathbf{I}-\frac{X}{\|X\|_F}\right)-Y\right\|_F\geq \left\|\lambda\mathbf{I}-Y\right\|_F-\lambda \geq \lambda(\sqrt{r}-1).
 \]
\end{proof}

\section{Proof of Theorem \ref{th:rank1mod}}
The aim of this section is to present the proof of Theorem \ref{th:rank1mod}.


\begin{proof}[Proof of Theorem \ref{th:rank1mod}]
Set
 \begin{equation}
\label{eq: rank_1app}
\vx^\#\,\,:=\,\,\argmin{\vx\in \mathbb{C}^d} \mu\|\vx\|_1^2+\frac{1}{2}\sum_{i=1}^m(|\langle \va_i,\vx\rangle|^2-b_i)^2.
\end{equation}
Noting $\exp(i\theta)\vx^\#$ is also a solution to (\ref{eq: rank_1app}) for any
$\theta\in \R$, without loss of generality, we  can assume that
\[
\innerp{\vx^\#,\vx_{0}}\in \R\qquad\text{and}\qquad\innerp{\vx^\#,\vx_{0}}\geq 0.
\]
Then a simple observation  is that $X^{\#}$ is the solution to (\ref{eqn: unconstrained_lowrank}) if and only if $X^\#=\vx^\# (\vx^\#)^*$.

Set $X_0:=\vx_0\vx_0^*$ and
\[
H:=X^\#-X_0=\vx^\#(\vx^\#)^{*}-\vx_{0}\vx_{0}^{*}.
\]
 To prove the conlusion, it is enough to
consider the upper bound of $\|H\|_F$.  Set $T_{0}:=\text{supp}(\vx_{0})$.
 Set $T_{1}$ as the index set which contains the indices of the $ak$ largest
 elements of  $\vx^\#_{T_{0}^{c}}$ in magnitude, and $T_2$ contains the indices of the next $ak$ largest elements, and so on. For simplicity, we set $T_{01}:=T_{0}\cup T_{1}$
and  $\bar{H}:=H_{T_{01},T_{01}}$. Note that
\[
 \|H\|_F\leq \|{\bar H}\|_F+
\sum_{i\geq2,j\geq2}\|H_{T_{i},T_{j}}\|_{F}+2\sum_{j\geq2,
i=0,1}\|H_{T_{i},T_{j}}\|_{F}.
\]
So, it is enough to present  upper bounds for \[\|{\bar H}\|_F,
\sum_{i\geq2,j\geq2}\|H_{T_{i},T_{j}}\|_{F},\  \text{and} \ \sum_{j\geq2,
i=0,1}\|H_{T_{i},T_{j}}\|_{F}.\] We first consider
$\sum_{i\geq2,j\geq2}\|H_{T_{i},T_{j}}\|_{F}$.
  According to
\[
\mu \|X^\#\|_1+\frac{1}{2}\|\mathcal{A}(X^\#)-\vb\|_2^2\leq \mu \|X_0\|_1+\frac{1}{2}\|\mathcal{A}(X_0)-\vb\|_2^2,
\]
we can obtain that
\begin{equation}
\label{eqn: inter1}
\mu\|H-H_{T_0,T_0}\|_1\leq \mu\|H_{T_0,T_0}\|_1+\|\vw\|_2\|\mathcal{A}(H)\|_2-\frac{1}{2}\|\mathcal{A}(H)\|_2^2.
\end{equation}
Here, we use
\[
\|X_0\|_1 = \|H_{T_0,T_0}-X^\#_{T_0,T_0}\|_1\leq \|H_{T_0,T_0}\|_1+\|X^\#_{T_0,T_0}\|_1,
\]
and
\[
\|X^\#\|_1-\|X_{T_0,T_0}^\#\|_1 = \|X^\#-X^\#_{T_0,T_0}\|_1= \|H-H_{T_0,T_0}\|_1.
\]
 Therefore, we have
\begin{equation}\label{eq:term1}
\begin{aligned}
 \sum_{i\geq2,j\geq2}\|H_{T_{i},T_{j}}\|_{F}
& = \sum_{i\geq2,j\geq2}\|\vx^\#_{T_{i}}\|_{2}\cdot \|\vx^\#_{T_{j}}\|_{2}
=\left(\sum_{i\geq2}\|\vx^\#_{T_{i}}\|_{2}\right)^{2}\\
&\leq\frac{1}{ak}\|\vx^\#_{T_{0}^{c}}\|_{1}^{2}=\frac{1}{ak}\|H_{T_{0}^{c},T_{0}^{c}}\|_{1}\leq
\frac{1}{ak} \|H-H_{T_{0},T_{0}}\|_{1}\\
& \leq \frac{1}{ak}\|H_{T_0,T_0}\|_1+\frac{\|\vw\|_2}{\mu ak}\|\mathcal{A}(H)\|_2-\frac{1}{2\mu ak}\|\mathcal{A}(H)\|_2^2\\
& \leq \frac{1}{a}\|H_{T_0,T_0}\|_F+\frac{\|\vw\|_2}{\mu ak}\|\mathcal{A}(H)\|_2-\frac{1}{2\mu ak}\|\mathcal{A}(H)\|_2^2\\
 & \leq \frac{1}{a}\|\bar{H}\|_F+\frac{\|\vw\|_2}{\mu ak}\|\mathcal{A}(H)\|_2-\frac{1}{2\mu ak}\|\mathcal{A}(H)\|_2^2.
 \end{aligned}
\end{equation}
The second line based on $\|\vx^\#_{T_j}\|_2\leq \|\vx^\#_{T_{j-1}}\|_1/\sqrt{ak}$, and the  third line follows from (\ref{eqn: inter1}).

Second, we consider  $\sum_{j\geq 2, i=0,1}\|H_{T_{i},T_{j}}\|_{F}$. Applying that
$\|\vx^\#_{T_j}\|_2\leq \|\vx^\#_{T_{j-1}}\|_1/\sqrt{ak}$, we have
\begin{equation}\label{eq:term2}
\begin{aligned}
\sum_{j\geq2}\|H_{T_{i},T_{j}}\|_{F}&=\|\vx^\#_{T_{i}}\|_{2} \cdot\sum_{j\geq2}\|\vx^\#_{T_{j}}\|_{2}
\leq\frac{1}{\sqrt{ak}}\|\vx^\#_{T_{0}^{c}}\|_{1}\|\vx^\#_{T_{i}}\|_{2}\\
& \leq \frac{1}{\sqrt{a}}\|\vx^\#_{T_{01}}-\vx_0\|_2\|\vx^\#_{T_{i}}\|_{2}+
\frac{\|\vw\|_2}{\sqrt{2\mu ak}}\|\vx_{T_i}^\#\|_2\\
&\leq \frac{\sqrt{2}}{\sqrt{a}}\|\vx^\#_{T_{01}}(\vx^\#_{T_{01}})^*-\vx_0\vx_0^*\|_F
+\frac{\|\vw\|_2}{\sqrt{2\mu ak}}\|\vx_{T_i}^\#\|_2\\
& = \frac{\sqrt{2}}{\sqrt{a}}\|\bar{H}\|_F+
\frac{\|\vw\|_2}{\sqrt{2\mu ak}}\|\vx_{T_i}^\#\|_2,
\end{aligned}
\end{equation}
where $i\in \{0,1\}$. Here, the third line follows from Lemma \ref{lem:
MatrixToVector} and the second line follows from
\begin{equation}\label{eq:claim1}
\|\vx^\#_{T_{0}^c}\|_{1} \leq \sqrt{k}\|\vx^\#_{T_{0}}-\vx_0\|_2+\frac{\|\vw\|_2}{\sqrt{2\mu}}.
\end{equation}
Indeed, noting that
\[
\mu \|\vx^\#(\vx^\#)^*\|_1+\frac{1}{2}\|\mathcal{A}(\vx^\#(\vx^\#)^*)-\vb\|_2^2\leq \mu \|\vx_0\vx_0^*\|_1+\frac{1}{2}\|\mathcal{A}(\vx_0(\vx_0)^*)-\vb\|_2^2,
\]
we obtain that
\[
\mu \|\vx^\#\|_1^2+\frac{1}{2}\|\mathcal{A}(\vx^\#(\vx^\#)^*)-\vb\|_2^2\leq \mu \|\vx_0\|_1^2+\frac{1}{2}\|\mathcal{A}(\vx_0(\vx_0)^*)-\vb\|_2^2
\]
which implies
\begin{equation}
\label{eqn: vector_bound}
\begin{aligned}
\|\vx^\#\|_1&\leq \sqrt{\|\vx_0\|_1^2+\frac{1}{2\mu}\|\mathcal{A}(\vx_0(\vx_0)^*)-\vb\|_2^2-\frac{1}{2\mu}\|\mathcal{A}(\vx^\#(\vx^\#)^*)-\vb\|_2^2}\\
&=\sqrt{\|\vx_0\|_1^2+\frac{1}{2\mu}\|\mathcal{A}(\vx_0(\vx_0)^*)-\vb\|_2^2-\frac{1}{2\mu}\|\mathcal{A}(\vx^\#(\vx^\#)^*)-\mathcal{A}(\vx_0(\vx_0)^*)+\mathcal{A}(\vx_0(\vx_0)^*)-\vb\|_2^2}\\
&=\sqrt{\|\vx_0\|_1^2-\frac{1}{2\mu}\|\mathcal{A}(\vx_0(\vx_0)^*-\vx^\#(\vx^\#)^*)\|_2^2-\frac{1}{\mu}\langle \mathcal{A}(\vx_0(\vx_0)^*-\vx^\#(\vx^\#)^*),\mathcal{A}(\vx_0(\vx_0)^*)-\vb\rangle}\\
&\leq \sqrt{\|\vx_0\|_1^2+\frac{1}{\mu}\|\mathcal{A}(H)\|_2\|\mathcal{A}(\vx_0(\vx_0)^*)-\vb\|_2-\frac{1}{2\mu}\|\mathcal{A}(H)\|_2^2} \\
&=\sqrt{\|\vx_0\|_1^2+\frac{\|\vw\|_2}{\mu}\|\mathcal{A}(H)\|_2-\frac{1}{2\mu}\|\mathcal{A}(H)\|_2^2}\\
&\leq  \|\vx_0\|_1+\sqrt{\max\{0,\frac{\|\vw\|_2}{\mu}\|\mathcal{A}(H)\|_2-\frac{1}{2\mu}\|\mathcal{A}(H)\|_2^2\}}\\
&\leq \|\vx_0\|_1+\frac{\|\vw\|_2}{\sqrt{2\mu}}.
\end{aligned}
\end{equation}
Therefore, we have
\[
\|\vx^\#_{T_{0}^c}\|_{1} \leq-\|\vx^\#_{T_{0}}\|_1+\|\vx_0\|_1+\frac{\|\vw\|_2}{\sqrt{2\mu}} \leq \|\vx^\#_{T_{0}}-\vx_0\|_1+\frac{\|\vw\|_2}{\sqrt{2\mu}}
 \leq \sqrt{k}\|\vx^\#_{T_{0}}-\vx_0\|_2+\frac{\|\vw\|_2}{\sqrt{2\mu}},
\]
which implies (\ref{eq:claim1}).
  Combing
(\ref{eq:term1}) and (\ref{eq:term2}), we have
\begin{equation}\label{eq:HB0}
\begin{aligned}
& \sum_{i\geq2,j\geq2}\|H_{T_{i},T_{j}}\|_{F}+2\sum_{j\geq2, i=0,1}\|H_{T_{i},T_{j}}\|_{F}\\
&\leq \left(\frac{1}{a}+\frac{4\sqrt{2}}{\sqrt{a}}\right)\|\bar{H}\|_F+\frac{2\|\vw\|_2}{\sqrt{\mu ak}}\|\vx^\#_{T_{01}}\|_{2}
 +\frac{\|\vw\|_2}{\mu ak}\|\mathcal{A}(H)\|_2-\frac{1}{2\mu ak}\|\mathcal{A}(H)\|_2^2\\
&\leq  \left(\frac{1}{a}+\frac{4\sqrt{2}}{\sqrt{a}}\right)\|\bar{H}\|_F+
 \frac{2\|\vw\|_2}{\sqrt{\mu ak}}\|\vx^\#_{T_{01}}\|_{2} +\frac{\|\vw\|_2^2}{2\mu ak}.
 \end{aligned}
\end{equation}
Third, we claim that
\begin{equation}\label{eq:HB}
\|{\bar H}\|_F\leq \frac{1}{c-\frac{C}{a}-\frac{4\sqrt{2}C}{\sqrt{a}}} \left(2C\frac{\|\vw\|_2}{\sqrt{\mu ak}}\|\vx^\#_{T_{01}}\|_{2}+\frac{\mu
ak}{2C}\left(\frac{C\|\vw\|_2 }{\mu ak}+\frac{1}{\sqrt{m}}\right)^2\right).
\end{equation}
Combining (\ref{eq:HB0}) and (\ref{eq:HB}), we obtain that
\[
\begin{split}
\|H\|_F&\leq  \|{\bar H}\|_F+\sum_{i\geq2,j\geq2}\|H_{T_{i},T_{j}}\|_{F}+2\sum_{j\geq2, i=0,1}\|H_{T_{i},T_{j}}\|_{F}\\
&\leq \left(1+\frac{1}{a}+\frac{4\sqrt{2}}{\sqrt{a}}\right)\|\bar{H}\|_F+\frac{2\|\vw\|_2}{\sqrt{\mu ak}}\|\vx^\#_{T_{01}}\|_{2}+\frac{\|\vw\|_2^2}{{2\mu ak}}\\
&\leq (2C\alpha+2)\frac{\|\vw\|_2}{\sqrt{\mu ak}}\|\vx^\#_{T_{01}}\|_{2}+\frac{\|\vw\|_2^2}{{2\mu ak}}+\alpha\cdot\frac{\mu ak}{2C}\left(\frac{C\|\vw\|_2 }{\mu ak}+\frac{1}{\sqrt{m}}\right)^2\\
&\leq (2C\alpha+2)\frac{\|\vw\|_2}{\sqrt{\mu ak}}\left(\|\vx_0\|_1+\frac{\|\vw\|_2}{\sqrt{2\mu}}\right)+\frac{\|\vw\|_2^2}{{2\mu ak}}+\alpha\cdot\frac{\mu ak}{2C}\left(\frac{C\|\vw\|_2 }{\mu ak}+\frac{1}{\sqrt{m}}\right)^2,
\end{split}
\]
which leads to the conclusion.  Here, the fourth line is based on
\[
\|\vx_{T_{01}}^\#\|_2\leq \|\vx_{T_{01}}^\#\|_1\leq  \|\vx^\#\|_1\leq \|\vx_0\|_1+\frac{\|\vw\|_2}{\sqrt{2\mu}},
\]
where the last inequality follows from  (\ref{eqn: vector_bound}).

We remain to prove (\ref{eq:HB}). Note that
 \[
 \frac{1}{\sqrt{m}} \|\A(H)\|_2\geq \frac{1}{m}\|\A(H)\|_1\geq \frac{1}{m}\|\A(\bar{H})\|_1-\frac{1}{m}\|\A(H-\bar{H})\|_1,
 \]
 which implies
 \begin{equation}\label{eq:AHup}
 \frac{1}{m}\|\A(\bar{H})\|_1\leq \frac{1}{m}\|\A(H-\bar{H})\|_1+\frac{1}{\sqrt{m}} \|\A(H)\|_2.
 \end{equation}
 Here we can see that
 \[
 \frac{1}{m}\|\A(H-\bar{H})\|_1\leq \sum_{i=0,1}\frac{1}{m}\|\A(H_{T_{i},T_{01}^c}+H_{T_{01}^c,T_{i}})\|_1+\frac{1}{m}\|\A(H_{T_{01}^c,T_{01}^c})\|_1.
 \]
 For $i=0, 1$, since $\mathcal{A}(\cdot)$ satisfy the RIP condition of order $(2,2ak)$ with
upper RIP constant $C$, we have
 \begin{equation}\label{eqn:mid1}
\begin{split}
&\frac{1}{m}\|\A(H_{T_{i},T_{01}^c}+H_{T_{01}^c,T_{i}})\|_1=\frac{1}{m}\left\|\sum_{j\geq 2}\A(H_{T_{i},T_{j}}+H_{T_{j},T_{i}})\right\|_1\\
&\leq \frac{1}{m}\sum_{j\geq 2}\|\mathcal{A}(H_{T_{i},T_{j}}+H_{T_{j},T_{i}})\|_1=\frac{1}{m}\sum_{j\geq 2}\|\mathcal{A}(\vx_{T_i}^\#(\vx_{T_j}^\#)^*+\vx_{T_j}^\#(\vx_{T_i}^\#)^*)\|_1\\&\leq 2C\sum_{j\geq 2}\|\vx_{T_{i}}^\#\|_2\|\vx_{T_j}^\#\|_2\leq \frac{2C}{\sqrt{ak}}\|\vx_{T_0^c}^\#\|_1\|\vx_{T_i}^\#\|_2\\
&\leq  \frac{2C}{\sqrt{a}}\|\bar{H}\|_F+\frac{2C\|\vw\|_2}{\sqrt{2\mu ak}}\|\vx^\#_{T_{i}}\|_{2}.
\end{split}
\end{equation}
Here, the last line follows from (\ref{eq:term2}). On the other hand, based on
(\ref{eq:term1}), we have
\begin{equation}\label{eqn:mid2}
 \begin{split}
 \frac{1}{m}\|\A(H_{T_{01}^c,T_{01}^c})\|_1&\leq \frac{1}{m}\left\|\sum_{i,j\geq 2, i\neq j}\A(H_{T_{i},T_{j}}+H_{T_{j},T_{i}})\right\|_1+\frac{1}{m}\left\|\sum_{i\geq 2}\A(H_{T_{i},T_{i}})\right\|_1\\
 &\leq C \sum_{i\geq2,j\geq2}\|H_{T_{i},T_{j}}\|_{F} \leq \frac{C}{a}\|\bar{H}\|_F+\frac{C\|\vw\|_2 }{\mu ak}\|\mathcal{A}(H)\|_2-\frac{C}{2\mu ak}\|\mathcal{A}(H)\|_2^2.
 \end{split}
 \end{equation}
 As $\mathcal{A}(\cdot)$ satisfy the RIP condition of order $(2,2ak)$ with lower RIP constant $c>0$, combining (\ref{eq:AHup}), (\ref{eqn:mid1}) and (\ref{eqn:mid2}), we obtain that
  \begin{equation}
 \begin{aligned}
 &c\|\bar{H}\|_F\leq \frac{1}{m}\|\A(\bar{H})\|_1\leq  \frac{1}{\sqrt{m}} \|\A(H)\|_2+\frac{1}{m}\|\A(H-\bar{H})\|_1\\
 &\leq  \frac{1}{\sqrt{m}} \|\A(H)\|_2+ \sum_{i=0,1}\frac{1}{m}\|\A(H_{T_{i},T_{01}^c}+H_{T_{01}^c,T_{i}})\|_1+\frac{1}{m}\|\A(H_{T_{01}^c,T_{01}^c})\|_1\\
 &\leq  \frac{1}{\sqrt{m}} \|\A(H)\|_2+2C\frac{\|\vw\|_2}{\sqrt{\mu ak}}(\|\vx^\#_{T_{0}}\|_{2}+\|\vx^\#_{T_{1}}\|_{2})+ \left(\frac{C}{a}+\frac{4\sqrt{2}C}{\sqrt{a}}\right)\|\bar{H}\|_F+\frac{C\|\vw\|_2 }{\mu ak}\|\mathcal{A}(H)\|_2-\frac{C}{2\mu ak}\|\mathcal{A}(H)\|_2^2\\
 &\leq  \frac{1}{\sqrt{m}} \|\A(H)\|_2+2C\frac{\|\vw\|_2}{\sqrt{\mu ak}}\|\vx^\#_{T_{01}}\|_{2}+ \left(\frac{C}{a}+\frac{4\sqrt{2}C}{\sqrt{a}}\right)\|\bar{H}\|_F+\frac{C\|\vw\|_2 }{\mu ak}\|\mathcal{A}(H)\|_2-\frac{C}{2\mu ak}\|\mathcal{A}(H)\|_2^2,
 \end{aligned}
 \end{equation}
which implies
 \begin{equation*}
 \begin{aligned}
\left(c-\frac{C}{a}-\frac{4\sqrt{2}C}{\sqrt{a}}\right) \|\bar{H}\|_F&\leq 2C\frac{\|\vw\|_2}{\sqrt{\mu ak}}\|\vx^\#_{T_{01}}\|_{2}+\left(\frac{C\|\vw\|_2 }{\mu ak}+\frac{1}{\sqrt{m}}\right)\|\mathcal{A}(H)\|_2-\frac{C}{2\mu ak}\|\mathcal{A}(H)\|_2^2\\
&\leq 2C\frac{\|\vw\|_2}{\sqrt{\mu ak}}\|\vx^\#_{T_{01}}\|_{2}+\frac{\mu ak}{2C}\left(\frac{C\|\vw\|_2 }{\mu ak}+\frac{1}{\sqrt{m}}\right)^2.
 \end{aligned}
\end{equation*}
It leads to the inequality (\ref{eq:HB}).
\end{proof}

\section{Proof of Lemma \ref{le:xia}}\label{sec: prove lemma}

Denote $\R^{d\times d}_{\rm sym}$ as the set of symmetric real $d\times d$ matrices,
and $\R^{d\times d}_{\rm skew}$ as  the set of skew-symmetric real $d\times d$
matrices.   If $X\in \mathbb{H}^{d\times d}$, then $X$ can be written as  $X=X_1+iX_2$, where $X_1\in \R^{d\times d}_{\rm sym}$ and $X_2\in
\R^{d\times d}_{\rm skew}$ are the real and imaginary parts of $X$.  Thus the set
$\{X\in \mathbb{H}^{d\times d}\ :\ X\succeq 0\}$ corresponds to
\[
\mathbb{H}_{+}^{d\times d}:=\left\{\begin{bmatrix} {X_1}\\ X_2 \end{bmatrix}:
\left( X_1, \\ X_2 \right)\in \R^{d\times d}_{\rm sym}\times \R^{d\times d}_{\rm skew},  \
 \vz_1^\mathrm{T} X_1\vz_1+\vz_2^\mathrm{T} X_1 \vz_2+\vz_2^\mathrm{T} X_2\vz_1-\vz_1^\mathrm{T} X_2 \vz_2\geq 0\ \text{for}\ \text{all}\ \vz_1,\vz_2\in \mathbb{R}^{d}\right\}.
\]
Let $\widetilde{\mathcal{A}}:\mathbb{R}^{2d\times d} \rightarrow \R^m$ be defined by
\begin{equation}\label{eqn: A}
 \begin{bmatrix} {X_1}\\ X_2 \end{bmatrix}\mapsto (\mathscr{R}(\va_i)^\mathrm{T} X_1\mathscr{R}(\va_i)+\mathscr{I}(\va_i)^\mathrm{T} X_1\mathscr{I}(\va_i)+\mathscr{I}(\va_i)^\mathrm{T} X_2\mathscr{R}(\va_i)-\mathscr{R}(\va_i)^\mathrm{T} X_2 \mathscr{I}(\va_i))_{i=1}^m.
\end{equation}
Then $\mathcal{A}(X)=\widetilde{A}\left(\begin{bmatrix} {X_1}\\ X_2
\end{bmatrix}\right)$. By a simple calculation,  its conjugate map
$\widetilde{\mathcal{A}}^*: \R^m\rightarrow\mathbb{R}^{2d\times d} $ is given by
\begin{equation}\label{eqn: A*}
(b_i)_{i=1}^m\mapsto \begin{bmatrix} {\sum_{i=1}^m b_i\left(\mathscr{R}(\va_i)\mathscr{R}(\va_i)^\mathrm{T}+\mathscr{I}(\va_i)\mathscr{I}(\va_i)^\mathrm{T}\right)}\\ \sum_{i=1}^m b_i\left(\mathscr{I}(\va_i)\mathscr{R}(\va_i)^\mathrm{T}-\mathscr{R}(\va_i)\mathscr{I}(\va_i)^\mathrm{T}\right) \end{bmatrix}.
\end{equation}
For $X=X_1+i X_2\in \mathbb{C}^{d\times d}$, $\|X\|_1$ and $\|X\|_F$ can also be written as
\[
\|X\|_1=\left\|\begin{bmatrix} \text{vec}{(X_1)}\\ \text{vec}(X_2) \end{bmatrix}\right\|_{1,2}=\sum_{i,j}\sqrt{[X_1]_{i,j}^2+[X_2]_{i,j}^2},\quad
\text{and}\quad \|X\|_F=\left\|\begin{bmatrix} {X_1}\\ X_2 \end{bmatrix}\right\|_{F}.
\]
Using the notations above, we recast the model (\ref{eq: unconstrained SDP}) as follows.
\begin{equation}\label{eq: unconstrained SDP equ}
\min\limits_{X_1,X_2} \lambda \left(\text{Tr}(X_1)-\left\|\begin{bmatrix} X_1 \\ X_2 \end{bmatrix}\right\|_{F}
\right)+\mu \left\|\begin{bmatrix} \text{vec}{(X_1)}\\ \text{vec}(X_2) \end{bmatrix}\right\|_{1,2} +\frac{1}{2}\left\|\widetilde{\mathcal{A}}\left(\begin{bmatrix} X_1 \\ X_2 \end{bmatrix}\right)-\vb\right\|_2^2\quad {\rm s.t.}\  \begin{bmatrix} X_1 \\ X_2 \end{bmatrix}\in \mathbb{H}_{+}^{d\times d}.
\end{equation}
If $(X_1^{\#}, X_2^{\#})$ is a minimizer of (\ref{eq: unconstrained SDP equ}), then
the optimal solution $X^{\#}$ of (\ref{eq: unconstrained SDP}) satisfies
$X^{\#}=X_1^{\#}+ iX_2^{\#}$.

In order to prove Lemma \ref{le:xia}, we first introduce some technical lemmas in convex optimization and matrix theory.
Assume that $\Omega\subset \mathbb{R}^n$. We use $T_\Omega(\vx^\#)$  and
$T_\Omega(\vx^\#)^*$ to denote the tangent cone of $\Omega$ at $\vx^\#\in \Omega$ and
its dual cone, respectively. Particularly, we have
\begin{prop}
\label{prop: opt2} If $\Omega$ is a convex cone in $\mathbb{R}^n$ and $\vx^\#\in
\Omega$, then
\[
T_\Omega(\vx^\#)^*=\{\vy\in \mathbb{R}^n\ : \ \langle \vy,\vx\rangle\leq 0 \ \text{for}\ \text{all}\ \vx\in {\Omega},
\ \text{and}\ \langle \vy,\vx^\#\rangle=0\}.
\]
\end{prop}
\begin{proof}
According to Proposition 4.6.3 in \cite{BNO03}, we have
\begin{equation}\label{eq:coned}
T_\Omega(\vx^\#)^*=\{\vy\in \mathbb{R}^n\ :\ \langle \vy,\vx-\vx^\#\rangle\leq 0 \
\text{for}\ \text{all}\ \vx\in {\Omega}\}.
\end{equation}
 Assume that $\vy\in T_\Omega(\vx^\#)^*$. Then $\langle \vy,\vx-\vx^\#\rangle\leq 0 \
\text{for}\ \text{all}\ \vx\in {\Omega}$. Since $\Omega$ is a cone, we have
$\vx^\#/2, 2\vx^\#\in \Omega$. Taking $\vx=2\vx^\#$, we obtain
$\innerp{\vy,\vx^\#}\leq 0$. Similarly, taking $\vx=\vx^\#/2$, we have
$\innerp{\vy,\vx^\#}\geq 0$. We arrive at $\innerp{\vy,\vx^\#}=0$, which  leads to
$\langle \vy,\vx\rangle\leq 0 \ \text{for}\ \text{all}\ \vx\in {\Omega}$.
\end{proof}

The following theorem provides some properties of local minimum on constrained model.
\begin{prop} \cite[Proposition 4.7.3]{BNO03}
\label{prop: opt1} Let $\vx^\#$ be a local minimizer of the model:
\[
\min_{\mathbf{x}\in \Omega} f_1(\mathbf{x})+f_2(\mathbf{x}),
\]
where $f_1$ is convex and $f_2$ is smooth over a  subset $\Omega$ of $\mathbb{R}^n$.
Assume that the tangent cone $T_{\Omega}(\vx^\#)$ is convex. Then
\[
-\nabla f_2(\vx^\#)\in \partial f_1(\vx^\#)+T_\Omega(\vx^\#)^*.
\]
\end{prop}

We next present the sub-gradient set of  $\partial(\|X\|_1)$:
\begin{prop}(\cite{B99}) \label{prop: subgradient}
Assume that $X=X_1+iX_2$ with $X_1,X_2\in {\mathbb R}^{d\times d}$. Then the
subgradient set of $\|X\|_1$ in real space is
\begin{equation}\label{eqn: l1_subgrad}
\begin{split}
  \partial \left(\left\|\begin{bmatrix}
\text{vec}{(X_1)}\\ \text{vec}(X_2) \end{bmatrix}\right\|_{1,2}\right)
:=\Bigg\{\begin{bmatrix} {G_1}\\ G_2 \end{bmatrix}, G_1,G_2\in \R^{d\times d} \ :\
&[G_1]_{i_1,i_2}^2+[G_2]_{i_1,i_2}^2\leq 1,\ \text{if} \  [X_1]_{i_1,i_2}=[X_2]_{i_1,i_2}=0;\\
& ([G_1]_{i_1,i_2},[G_2]_{i_1,i_2})=\frac{([X_1]_{i_1,i_2},[X_2]_{i_1,i_2})}{\sqrt{[X_1]_{i_1,i_2}^2+[X_2]_{i_1,i_2}^2}},
 \ \text{otherwise}\Bigg\}
\end{split}\end{equation}
\end{prop}
Combining Proposition \ref{prop: opt2}, Proposition
\ref{prop: opt1} with Proposition \ref{prop: subgradient}, we have
\begin{lem}
\label{lem: opt3} Assume that $(X_1^\#, X_2^\#)$ is a local minimizer of model (\ref{eq: unconstrained SDP equ}).
 Then there exist  $ \begin{bmatrix} {\Lambda_1}\\ \Lambda_2 \end{bmatrix}\in \mathbb{H}_+^{d\times d}$ and
  $ \begin{bmatrix} {G_1}\\ G_2 \end{bmatrix}\in \partial \left(\left\|\begin{bmatrix}
    \text{vec}{(X_1^\#)}\\ \text{vec}(X_2^\#)
    \end{bmatrix}\right\|_{1,2}\right)$
  such that the followings  hold:
  
\begin{enumerate}[(i)]
\item Stationary condition:
\begin{equation}
\label{eqn: stationary}
\lambda\left(\begin{bmatrix} \mathbf{I}\\\mathbf{0}\end{bmatrix}- \begin{bmatrix} X_1^\# \\ X_2^\# \end{bmatrix}\Big/\left\|\begin{bmatrix} X_1^\# \\ X_2^\# \end{bmatrix}\right\|_{F}\right)+\mu \begin{bmatrix} G_1\\G_2\end{bmatrix} +\widetilde{\mathcal{A}}^*\left(\widetilde{\mathcal{A}}\left(\begin{bmatrix} X_1^\# \\ X_2^\# \end{bmatrix}\right)-\vb\right)-\begin{bmatrix} \Lambda_1\\ \Lambda_2 \end{bmatrix}=0;
\end{equation}

\item Complementary slackness condition:
\begin{equation}
\label{eqn: complementary} \left\langle\begin{bmatrix} \Lambda_1\\ \Lambda_2 \end{bmatrix}, \begin{bmatrix} X_1^\# \\ X_2^\# \end{bmatrix}\right\rangle=0.
\end{equation}
\end{enumerate}
\end{lem}
\begin{proof}
Set
\begin{equation*}
\begin{aligned}
f_1(X_1, X_2)&:=\mu \left\|\begin{bmatrix} \text{vec}{(X_1)}\\ \text{vec}(X_2)
\end{bmatrix}\right\|_{1,2},\\
f_2(X_1,X_2)&:=\lambda \left(\text{Tr}(X_1)-\left\|\begin{bmatrix} X_1 \\ X_2
\end{bmatrix}\right\|_{F} \right)
+\frac{1}{2}\left\|\widetilde{\mathcal{A}}\left(\begin{bmatrix} X_1 \\ X_2
\end{bmatrix}\right)-\vb\right\|_2^2,
\end{aligned}
\end{equation*}
 and $\Omega:= \mathbb{H}_{+}^{d\times d}$. Then $f_1$ is convex and $f_2$
is smooth. Since $\Omega$ is convex, we obtain that  $T_\Omega\left(\begin{bmatrix} X_1^\#\\
X_2^\# \end{bmatrix}\right)$ is convex by Proposition 4.6.2 in \cite{BNO03}. According to Proposition \ref{prop: opt1},  there exists $-\begin{bmatrix}
\Lambda_1\\ \Lambda_2 \end{bmatrix}\in T_{\Omega}\left(\begin{bmatrix} X_1^\#\\
X_2^\# \end{bmatrix}\right)^*$ such that the stationary condition (\ref{eqn:
stationary}) holds. Furthermore, we can use  Proposition \ref{prop: opt2} to obtain the complementary slackness condition
(\ref{eqn: complementary}).

 We remain to prove $ \begin{bmatrix} {\Lambda_1}\\ \Lambda_2
\end{bmatrix}\in \mathbb{H}_+^{d\times d}$. Take
$T_1=\vt_1\vt_1^\mathrm{T}+\vt_2\vt_2^\mathrm{T}$ and $T_2=\vt_2\vt_1^\mathrm{T}-\vt_1\vt_2^\mathrm{T}$ for any fixed
$\vt_1,\vt_2\in \mathbb{R}^d$. Then  $\begin{bmatrix} T_1
\\ T_2 \end{bmatrix}\in \Omega$.
By the definition of  $T_{\Omega}\left(\begin{bmatrix} X_1^\#\\ X_2^\#
\end{bmatrix}\right)^*$ and Proposition \ref{prop: opt1}, we obtain that
\[
\left\langle-\begin{bmatrix} \Lambda_1\\
\Lambda_2 \end{bmatrix}, \begin{bmatrix} T_1 \\ T_2 \end{bmatrix}\right\rangle\leq 0,
\]
which implies
\begin{equation}\label{eqn: SDP condition}
\vt_1^\mathrm{T}
\Lambda_1\vt_1+\vt_2^\mathrm{T} \Lambda_1 \vt_2+\vt_2^\mathrm{T} \Lambda_2\vt_1-\vt_1^\mathrm{T}
\Lambda_2 \vt_2\geq 0\ \text{for}\ \text{any}\ \vt_1,\vt_2\in
\mathbb{R}^d.
\end{equation}

If $\left( \Lambda_1, \\ \Lambda_2 \right)\in \R^{d\times d}_{\rm sym}\times
\R^{d\times d}_{\rm skew}$, then we arrive at the conclusion. Otherwise, we can replace $\Lambda_1$ and $\Lambda_2$ by
\[
\widetilde{\Lambda}_1:=\frac{\Lambda_1+\Lambda_1^\mathrm{T}}{2}\ \text{and}\ \widetilde{\Lambda}_2:=\frac{\Lambda_2-\Lambda_2^\mathrm{T}}{2}.
\]
Noting that $(\widetilde{\Lambda}_1, \widetilde{\Lambda}_2)\in \R^{d\times d}_{\rm
sym}\times \R^{d\times d}_{\rm skew}$ and
\[
 \left\langle-\begin{bmatrix} \widetilde{\Lambda}_1\\
\widetilde{\Lambda}_2 \end{bmatrix}, \begin{bmatrix}
T_1 \\ T_2 \end{bmatrix}\right\rangle=\left\langle-\begin{bmatrix} \Lambda_1\\
\Lambda_2 \end{bmatrix}, \begin{bmatrix} T_1 \\ T_2 \end{bmatrix}\right\rangle\leq 0,
\]
 we obtain that
 $ \begin{bmatrix} {\widetilde{\Lambda}_1}\\ \widetilde{\Lambda}_2 \end{bmatrix}\in \mathbb{H}_+^{d\times d}$.
 After a simple calculation, we also have
\[
\lambda\left(\begin{bmatrix} \mathbf{I}\\\mathbf{0}\end{bmatrix}- \begin{bmatrix} X_1^\# \\ X_2^\# \end{bmatrix}\Big/\left\|\begin{bmatrix} X_1^\# \\ X_2^\# \end{bmatrix}\right\|_{F}\right)+\mu \begin{bmatrix} \widetilde{G}_1\\ \widetilde{G}_2\end{bmatrix} +\widetilde{\mathcal{A}}^*\left(\widetilde{\mathcal{A}}\left(\begin{bmatrix} X_1^\# \\ X_2^\# \end{bmatrix}\right)-\vb\right)-\begin{bmatrix} \widetilde{\Lambda}_1\\ \widetilde{\Lambda}_2 \end{bmatrix}=0,
\]
and
 \[
 \left\langle\begin{bmatrix} \widetilde{\Lambda}_1\\ \widetilde{\Lambda}_2 \end{bmatrix}, \begin{bmatrix} X_1^\# \\ X_2^\# \end{bmatrix}\right\rangle= \left\langle\begin{bmatrix} \Lambda_1\\ \Lambda_2 \end{bmatrix}, \begin{bmatrix} X_1^\# \\ X_2^\# \end{bmatrix}\right\rangle=0,
 \]
where  $\widetilde{G}_1:=\frac{G_1+G_1^\mathrm{T}}{2},
\widetilde{G}_2:=\frac{G_2-G_2^\mathrm{T}}{2}$ and
\[
  \begin{bmatrix} {\widetilde{G}_1}\\ \widetilde{G}_2 \end{bmatrix}\in \partial \left(\left\|\begin{bmatrix}
    \text{vec}{(X_1^\#)}\\ \text{vec}(X_2^\#)
    \end{bmatrix}\right\|_{1,2}\right).
  \]
 Therefore, the stationary condition (\ref{eqn:
stationary}) and  complementary slackness condition   (\ref{eqn: complementary}) also hold
for
  $\begin{bmatrix} {{\Lambda}_1}\\
{\Lambda}_2 \end{bmatrix}:=\begin{bmatrix} {\widetilde{\Lambda}_1}\\
\widetilde{\Lambda}_2 \end{bmatrix}$.
\end{proof}

We next present the proof of Lemma \ref{le:xia}.
\begin{proof}[Proof of Lemma \ref{le:xia}]
Since $\frac{1}{2}\|\vb\|_2^2> \mu \|\vx_0\|_1^2+\frac{1}{2}\|\vw\|_2^2$, we obtain
that $X^\#\neq 0$.

We next consider the equivalent model (\ref{eq: unconstrained SDP equ}) with global minimizer $(X_1^\#, X_2^\#)$.
 According to Lemma {\ref{lem: opt3}},  there exist
  $\begin{bmatrix} {\Lambda_1}\\ \Lambda_2 \end{bmatrix}\in \mathbb{H}_+^{d\times d}$
   and $\begin{bmatrix} G_1\\G_2\end{bmatrix} \in \partial \left(\left\|\begin{bmatrix} \text{vec}{(X_1^\#)}\\ \text{vec}(X_2^\#) \end{bmatrix}\right\|_{1,2}\right)$ such that the following holds:
    \begin{equation}
\label{eqn: stationary1}
\lambda\left(\begin{bmatrix} \mathbf{I}\\\mathbf{0}\end{bmatrix}- \begin{bmatrix} X_1^\# \\ X_2^\# \end{bmatrix}\Big/\left\|\begin{bmatrix} X_1^\# \\ X_2^\# \end{bmatrix}\right\|_{F}\right)+\mu \begin{bmatrix} G_1\\G_2\end{bmatrix} +\widetilde{\mathcal{A}}^*\left(\widetilde{\mathcal{A}}\left(\begin{bmatrix} X_1^\# \\ X_2^\# \end{bmatrix}\right)-\vb\right)-\begin{bmatrix} \Lambda_1\\ \Lambda_2 \end{bmatrix}=0;
\end{equation}
and
\begin{equation}
\label{eqn: complementary1} \left\langle\begin{bmatrix} \Lambda_1\\ \Lambda_2 \end{bmatrix}, \begin{bmatrix} X_1^\# \\ X_2^\# \end{bmatrix}\right\rangle=0.
\end{equation}
According to (\ref{eqn: stationary1}), we obtain that
\begin{equation}\label{eqn: norm_est}
\begin{aligned}
\left\|\widetilde{\mathcal{A}}^*\left(\widetilde{\mathcal{A}}\left(\begin{bmatrix} X_1^\# \\ X_2^\# \end{bmatrix}\right)-\vb\right)+\mu \begin{bmatrix} G_1\\G_2\end{bmatrix}\right\|_F
&=\left\|\lambda\left(\begin{bmatrix} \mathbf{I}\\\mathbf{0}\end{bmatrix}- \begin{bmatrix} X_1^\# \\ X_2^\# \end{bmatrix}\Big/\left\|\begin{bmatrix} X_1^\# \\ X_2^\# \end{bmatrix}\right\|_{F}\right)-\begin{bmatrix} \Lambda_1\\ \Lambda_2 \end{bmatrix}\right\|_F\\
& =\left\|\lambda\left(\mathbf{I}-\frac{X^\#}{\|X^\#\|_F}\right)-\Lambda\right\|_F\\
&\geq \lambda(\sqrt{r}-1),
\end{aligned}
\end{equation}
where $\Lambda:=\Lambda_1+i\Lambda_2 \in {\mathbb C}^{d\times d}$ and $r:={\rm
rank}(X^\#)$. The last inequality in (\ref{eqn: norm_est}) follows from (\ref{eqn: complementary1}) and Lemma
\ref{lem: SDP_property}.

On the other hand, we have
\begin{equation}
\label{eqn: upper}
\begin{aligned}
\left\|\widetilde{\mathcal{A}}^*\left(\widetilde{\mathcal{A}}\left(\begin{bmatrix} X_1^\# \\ X_2^\# \end{bmatrix}\right)-\vb\right)+\mu \begin{bmatrix} G_1\\G_2\end{bmatrix}\right\|_F&\leq \left\|\widetilde{\mathcal{A}}^*\left(\widetilde{\mathcal{A}}\left(\begin{bmatrix} X_1^\# \\ X_2^\# \end{bmatrix}\right)-\vb\right)\right\|_F+\mu\left\|\begin{bmatrix} G_1\\G_2\end{bmatrix}\right\|_F\\
&\leq  \|\mathcal{A}\|\left\|\widetilde{\mathcal{A}}\left(\begin{bmatrix} X_1^\# \\ X_2^\# \end{bmatrix}\right)-\vb\right\|_2+\mu d\\
&=\|\mathcal{A}\|\|\mathcal{A}(X^\#)-\vb\|_2+\mu d\\
&\leq \|\mathcal{A}\|\sqrt{2\mu \|\vx_0\|_1^2+\|\vw\|_2^2}+\mu d\\
&\leq \|\mathcal{A}\|(\sqrt{2\mu}\|\vx_0\|_1+\|\vw\|_2)+\mu d.
\end{aligned}
\end{equation}
Here, the second inequality follows from  Proposition \ref{prop: subgradient} and $[G_1]_{i_1,i_2}^2+[G_1]_{i_1,i_2}^2\leq 1$ for any $i_1,i_2\in \{1,...,d\}$.
Combing (\ref{eqn: norm_est}) and (\ref{eqn: upper}), we obtain that
\[
\lambda({\sqrt{r}-1})\leq {\mu d+\|\mathcal{A}\|(\sqrt{2\mu}\|\vx_0\|_1+\|\vw\|_2)}.
\]
By the assumption on $\lambda$ in (\ref{eq:lam}) as
\[
\lambda >\frac{ \mu d+\|\mathcal{A}\|(\sqrt{2\mu}\|\vx_0\|_1+\|\vw\|_2)}{\sqrt{2}-1},
\]
we have
\[
\frac{\sqrt{r}-1}{\sqrt{2}-1}\left({ \mu d+\|\mathcal{A}\|(\sqrt{2\mu}\|\vx_0\|_1+\|\vw\|_2)}\right)\leq {\mu d+\|\mathcal{A}\|(\sqrt{2\mu}\|\vx_0\|_1+\|\vw\|_2)}.
\]
Thus the only proper choice of $r$ is $r=1$.
\end{proof}

\section{Algorithms for solving Sparse PhaseLiftOff }

\subsection{The DCA algorithm}
In this section,  we establish an algorithm to  solve the Sparse PhaseLiftOff model
(\ref{eq: unconstrained SDP}), which is stated in Algorithm \ref{alg1}.
 \begin{algorithm}[t]
  \caption{The DCA Algorithm for solving  model (\ref{eq: unconstrained SDP})}
  \label{alg1}
  \begin{algorithmic}[1]
\State {\bf Input:} the map ${\mathcal A}$, the vector $\vb$, the tolerance error
${\rm tol}\geq 0$, the parameters $\lambda, \mu$ and MAXiter.
     \State {\bf Output:} A matrix $X^\#$.
    \State {\bf Initial:} ${X}^{0}=\mathbf{0}$.
  \State{\bf Loop:} {\bf for} $k=0$ {\bf to} MAXiter

  $
{Y}^{k}=\begin{cases}
\text{\ensuremath{\frac{{X}^{k}}{\|{X}^{k}\|_{F}}}} & \text{if}\ {X}^{k}\neq\mathbf{0}\\
\mathbf{0} & \text{if}\ {X}^{k}=\mathbf{0}
\end{cases}
$

\begin{equation}\label{eq: xk_iteration}
{X}^{k+1}=\underset{{X}\succeq0}{\text{argmin}}\Big\{\frac{1}{2}\|\mathcal{A}({X})-{\mathbf{b}}\|_{2}^{2}+\lambda\text{Tr}(X)-\lambda\langle {X},{Y}^{k}\rangle+\mu\|{X}\|_1\Big\}
\end{equation}
    {\bf If} {$\frac{\|{X}^{k}-{X}^{k-1}\|_{F}}{\text{\ensuremath{\max}}\{\|{X}^{k}\|_{F},1\}}\leq \text{tol}$} {\bf then} {\bf break}
\State $X^\#=X^k$.
  \end{algorithmic}
\end{algorithm}
Our algorithm is based on DCA, which is a descent method introduced by Tao and An
\cite{TA97,TA88}. DCA is also studied in compressed sensing, and in matrix recovery
problem (see \cite{XL18,phaseliftoff,YLX15}).

The step 6 of  Algorithm \ref{alg1} is to solve a subproblem (\ref{eq: xk_iteration}). We
suggest  employing ADMM method  \cite{BP11} to solve it, which is shown in Algorithm 2.    The convergence rate of ADMM was established in \cite{HY12}.
To derive  ADMM, we rewrite (\ref{eq: xk_iteration}) as
 \begin{equation}\label{eq:X123}
\underset{{X_3}\succeq0, X_3=X_1, X_3=X_2}{\text{min}}\frac{1}{2}\|\mathcal{A}({X_1})-{\mathbf{b}}\|_{2}^{2}+\lambda\text{Tr}(X_1)-\lambda\langle {X_1},{Y}^{k}\rangle+\mu\|{X_2}\|_1.
 \end{equation}
 The problem  (\ref{eq:X123}) is called global consensus problem  \cite[Equation (7.2)]{BP11} with local variables $X_1$ and $X_2$ and a common global variable $X_3$.
The augmented Lagrangian function corresponding to (\ref{eq:X123}) is
 \[
 \begin{split}
 \mathcal{L}_{\delta}(X_1,X_2,X_3,Y_1,Y_2)=&\frac{1}{2}\|\mathcal{A}({X_1})-{\mathbf{b}}\|_{2}^{2}
 +\langle X_1, \lambda(\mathbf{I}-Y^k)\rangle+\mu\|X_2\|_1+g_\succeq(X_3)\\
 &+\langle Y_1, X_1-X_3\rangle+\langle Y_2, X_2-X_3\rangle
+\frac{\delta}{2}\|X_1-X_3\|_F^2+\frac{\delta}{2}\|X_2-X_3\|_F^2,
 \end{split}
 \]
 where $Y_1,Y_2$ are dual variables, $\delta$ is augmented Lagrangian parameter
and
 \[
 g_{\succeq 0}(Z)=\begin{cases}
0 & \text{if}\ Z\succeq 0,\\
\infty & \text{otherwise}.
\end{cases}
 \]
 We can employ the standard  ADMM to solve
 \begin{equation}\label{eq:augX123}
\min_{X_1,X_2,X_3,Y_1,Y_2}  \mathcal{L}_{\delta}(X_1,X_2,X_3,Y_1,Y_2),
 \end{equation}
 which consists of updating  on both the primal and   dual variables \cite[Equation (7.3)-Equation (7.5)]{BP11}:
  \begin{equation}\label{eqn: ADMM}
 \begin{cases}
X_1^{l+1}=\arg\min_{X_1}  \mathcal{L}_{\delta}(X_1,X_2^l,X_3^l,Y_1^l,Y_2^l)\\
X_2^{l+1}=\arg\min_{X_2}  \mathcal{L}_{\delta}(X_1^{l+1},X_2,X_3^l,Y_1^l,Y_2^l)\\
X_3^{l+1}=\arg\min_{X_3}  \mathcal{L}_{\delta}(X_1^{l+1},X_2^{l+1},X_3,Y_1^l,Y_2^l)\\
Y_1^{l+1}=Y_1^{l}+\delta(X_1^{l+1}-X_3^{l+1})\\
Y_2^{l+1}=Y_2^{l}+\delta(X_2^{l+1}-X_3^{l+1})\end{cases}
 \end{equation}
 According to \cite{BP11}, $\delta$ can be fixed or adaptively
 updated following the rules below:
 \[
 \delta^{l+1}=\begin{cases}
2\delta^{l} & \text{if}\ \|R^{l}\|_F>10\|S^{l}\|_F\\
\delta^{l}/2 & \text{if}\  \|R^{l}\|_F<\frac{1}{10}\|S^{l}\|_F\\
\delta^{l} &\text{otherwise}\end{cases},
 \]
 where $\|R^{l}\|_F^2=\|X_1^l-X_3^l\|_F^2+\|X_2^l-X_3^l\|_F^2$, and $\|S^{l}\|_F^2=2(\delta^l)^2\|X_3^l-X_3^{l-1}\|_F^2$.

 More explicitly, we  state ADMM algorithm for solving (\ref{eqn: ADMM})  in Algorithm 2.   In
Algorithm 2, we use  $\mathcal{S}_{\lambda}: \mathbb{C}^{n\times n}\rightarrow
\mathbb{C}^{n\times n}$ to denote the soft-thresholding operator on each elements of
the matrix, i.e.,
\[
[\mathcal{S}_{\lambda}(Z)]_{i,j}=\begin{cases} {(|Z_{i,j}|-\lambda)\frac{Z_{i,j}}{|Z_{i,j}|}} & |Z_{i,j}|\geq \lambda,\\
0 &\text{otherwise}.\end{cases}
\]
   We use $\mathcal{P}_{\succeq}: \mathbb{H}^{n\times n}\rightarrow \mathbb{H}^{n\times n}$
    to denote the projection on the the positive semidefinite cone, i.e.,
    \[
    \mathcal{P}_{\succeq}(X)=U\max\{\Sigma, \mathbf{0}\}U^*,
    \]
    where $X=U\Sigma U^*$ is the eigenvalue decomposition of $X$.


\begin{algorithm}[t]
  \caption{ADMM for solving the subproblem (\ref{eq: xk_iteration})}
  \label{alg2}
  \begin{algorithmic}[1]
  \State {\bf Input:} the map ${\mathcal A}$, the vector $\vb$, $k$, $W=\lambda(\mathbf{I}-Y^k)$, the parameters $\lambda, \mu$, $\delta$ and MAXiter.
     \State {\bf Output:} A matrix $X^{k+1}$.
    \State {\bf Initial:} ${X}_{1}^0={X}_{2}^0={X}_{3}^0={Y}_{1}^0={Y}_{2}^0=\mathbf{0}$.
    \State{\bf Loop:} {\bf for} $l=0$ {\bf to} MAXiter

      $
      X_1^{l+1}=(\mathcal{A}^*\mathcal{A}+\delta\mathbf{I})^{-1}(\mathcal{A}^*(\mathbf{b})-W+\delta X_3^{l}-Y_1^{l})
      $

$ X_2^{l+1}=\mathcal{S}_{\mu/\delta}(X_3^{l}-\frac{1}{\delta}Y_2^{l})$

$
X_3^{l+1}=\mathcal{P}_{\succeq}\left(\frac{1}{2}(X_1^{l+1}+X_2^{l+1})+\frac{1}{2\delta}(Y_1^{l}+Y_2^{l})\right)
$

$ Y_1^{l+1}=Y_1^{l}+\delta(X_1^{l+1}-X_3^{l+1}) $

$ Y_2^{l+1}=Y_2^{l}+\delta(X_2^{l+1}-X_3^{l+1}) $
     \State
 $X^{k+1}=X_3^{l}.$
  \end{algorithmic}
\end{algorithm}

\subsection{The Convergence property of Algorithm \ref{alg1}}
The aim of this subsection is to study the convergence property of Algorithm \ref{alg1}.
Motivated by the techniques developed in \cite{phaseliftoff} and \cite{YLX15}, we
will show that Algorithm \ref{alg1} converges to a stationary point.
For convenience, we set
\[
  F(X):=\lambda (\text{Tr}(X)-\|X\|_F)+\mu \|X\|_1+\frac{1}{2}\|\mathcal{A}(X)-\vb\|_2^2.
  \]
 We first show that $\{F(X^{k})\}_{k\geq 1}$ generated by Algorithm \ref{alg1} is    a monotonically decreasing sequence.
 \begin{lem}\label{eqn: Fk decreasing}
 If $\{X^k\}_{k\geq 1}$ is a sequence generated by Algorithm \ref{alg1}, then we have
 \[
 F(X^k)-F(X^{k+1})\geq 0, \quad  \text{ for all } k\geq 0.
 \]
 \end{lem}
\begin{proof}
We consider the $k$th iteration of Algorithm \ref{alg1}. Recall that $X^{k+1}$ is the solution
to (\ref{eq: xk_iteration}) in Algorithm \ref{alg1}.
 Set $X^{k+1}:=X_1^{k+1}+i
X_2^{k+1}$ and $Y^k:=Y_1^k+i Y_2^k$ where $X_1^{k+1}, X_2^{k+1}, Y_1^k, Y_2^k \in
\mathbb{R}^{d\times d}$. Take
\begin{equation*}
\begin{aligned}
f_1(X_1, X_2)&:=\mu \left\|\begin{bmatrix} \text{vec}{(X_1)}\\ \text{vec}(X_2)
\end{bmatrix}\right\|_{1,2},\\
f_2(X_1,X_2)&:=\lambda \left(\text{Tr}(X_1)-\left\langle  \begin{bmatrix} {Y_1^{k}} \\ {Y_2^{k}}\end{bmatrix},\begin{bmatrix} X_1 \\ X_2
\end{bmatrix}\right\rangle\right)
+\frac{1}{2}\left\|\widetilde{\mathcal{A}}\left(\begin{bmatrix} X_1 \\ X_2
\end{bmatrix}\right)-\vb\right\|_2^2,
\end{aligned}
\end{equation*}
 and $\Omega:= \mathbb{H}_{+}^{d\times d}$. Then  $f_1$ is convex, $f_2$
is smooth, and  $T_\Omega\left(\begin{bmatrix} X_1^{k+1}\\
X_2^{k+1} \end{bmatrix}\right)$ is convex. According to Proposition \ref{prop: opt1}, we have
\begin{equation}\label{eqn: converge1}
\lambda\left(\begin{bmatrix} \mathbf{I}\\\mathbf{0}\end{bmatrix}- \begin{bmatrix} {Y_1^{k}} \\ {Y_2^{k}} \end{bmatrix}\right)+\mu \begin{bmatrix} G_1^{k+1}\\ G_2^{k+1}\end{bmatrix} +\widetilde{\mathcal{A}}^*\left(\widetilde{\mathcal{A}}\left(\begin{bmatrix} {X_1^{k+1}} \\ {X_2^{k+1}} \end{bmatrix}\right)-\vb\right)=\begin{bmatrix} {\Lambda_1^{k+1}}\\ {\Lambda_2^{{k+1}}} \end{bmatrix},
\end{equation}
and
\begin{equation}\label{eqn:converge2}
\left\langle\begin{bmatrix} {\Lambda_1^{k+1}}\\ {\Lambda_2^{k+1}} \end{bmatrix}, \begin{bmatrix} {X_1^{k+1}} \\ {X_2^{k+1}} \end{bmatrix}\right\rangle=0,
\end{equation}
for some $\Lambda^{k+1}=\Lambda_1^{k+1}+i \Lambda_2^{k+1}$ with  $\begin{bmatrix} {{\Lambda_1^{k+1}}}\\ {\Lambda_2^{k+1}} \end{bmatrix}\in \mathbb{H}_+^{d\times d}$, and $G^{k+1}=G_1^{k+1}+i G_2^{k+1}$
   with $\begin{bmatrix} {G_1^{k+1}}\\ {G_2^{k+1}}\end{bmatrix} \in \partial \left(\left\|\begin{bmatrix} \text{vec}{({X_1^{k+1}})}\\ \text{vec}({X_2^{k+1}}) \end{bmatrix}\right\|_{1,2}\right)$.
According to Proposition \ref{prop: subgradient}, we have
   \[
   \begin{cases}
   [G_1^{k+1}]_{i_1,i_2}^2+[G_2^{k+1}]_{i_1,i_2}^2\leq 1,\ &\text{if} \  [X_1^{k+1}]_{i_1,i_2}=[X_2^{k+1}]_{i_1,i_2}=0;\\
 ([G_1^{k+1}]_{i_1,i_2},[G_2^{k+1}]_{i_1,i_2})=\frac{([X_1^{k+1}]_{i_1,i_2},[X_2^{k+1}]_{i_1,i_2})}{\sqrt{[X_1^{k+1}]_{i_1,i_2}^2+[X_2^{k+1}]_{i_1,i_2}^2}},
 \ & \text{otherwise.}\end{cases}
   \]
Using a similar method for proving Lemma \ref{lem: opt3}, we can obtain
(\ref{eqn:converge2}).
 According to  (\ref{eqn: converge1}),  we have
 \begingroup\fontsize{8pt}{9pt}\selectfont
 \begin{equation}\label{eqn: inner estimation}
 \left\langle \begin{bmatrix} {X_1^k-X_1^{k+1}} \\ {X_2^k-X_2^{k+1}} \end{bmatrix}, \lambda\left(\begin{bmatrix} \mathbf{I}\\\mathbf{0}\end{bmatrix}- \begin{bmatrix} {Y_1^k} \\ {Y_2^k} \end{bmatrix}\right)+\mu \begin{bmatrix} G_1^{k+1}\\ G_2^{k+1}\end{bmatrix} +\widetilde{\mathcal{A}}^*\left(\widetilde{\mathcal{A}}\left(\begin{bmatrix} {X_1^{k+1}} \\ {X_2^{k+1}} \end{bmatrix}\right)-\vb\right)\right\rangle= \left\langle \begin{bmatrix} {X_1^k-X_1^{k+1}} \\ {X_2^k-X_2^{k+1}} \end{bmatrix},\begin{bmatrix} {\Lambda_1^{k+1}}\\ {\Lambda_2^{k+1}} \end{bmatrix}\right\rangle.
 \end{equation}
 \endgroup
 Combining   (\ref{eqn: inner estimation}) and
 \[
 \left\langle \begin{bmatrix} {X_1^{k+1}}\\ {X_2^{k+1}}\end{bmatrix},  \begin{bmatrix} {G_1^{k+1}}\\ {G_2^{k+1}}\end{bmatrix}\right\rangle=\|X^{k+1}\|_1, \ \  \left\langle \begin{bmatrix} {X_1^{k}}\\ {X_2^{k}}\end{bmatrix},  \begin{bmatrix} {Y_1^{k}}\\ {Y_2^{k}}\end{bmatrix}\right\rangle=\|X^{k}\|_F,\ \ \left\langle\begin{bmatrix} {\Lambda_1^{k+1}}\\ {\Lambda_2^{k+1}} \end{bmatrix}, \begin{bmatrix} {X_1^{k+1}} \\ {X_2^{k+1}} \end{bmatrix}\right\rangle=0,
 \]
 we obtain that
 \begin{equation}\label{eqn: inner mid}
 \begin{split}
\langle X^{k}, \Lambda^{k+1}\rangle=& \lambda\text{Tr}(X^k-X^{k+1})-\lambda\|X^k\|_F+\lambda\langle X^{k+1},Y^k\rangle\\
&+\mu\langle X^{k}, G^{k+1}\rangle-\mu\|X^{k+1}\|_1+\langle \mathcal{A}(X^k-X^{k+1}), \mathcal{A}(X^{k+1})-\vb\rangle,
\end{split}
 \end{equation}
since   $\widetilde{\mathcal{A}}\left(\begin{bmatrix} {X_1^{k+1}} \\ {X_2^{k+1}}
\end{bmatrix}\right)=\mathcal{A}(X^{k+1})$ and
$\widetilde{\mathcal{A}}\left(\begin{bmatrix} {X_1^{k}-X_1^{k+1}} \\
{X_2^{k}-X_2^{k+1}} \end{bmatrix}\right)=\mathcal{A}(X^{k}-X^{k+1})$ with
$X^k=X_1^k+iX_2^k$ and $X^{k+1}=X_1^{k+1}+iX_2^{k+1}$. Combining
\[
\begin{split}
F(X^k)-F(X^{k+1})=&\frac{1}{2}\|\mathcal{A}(X^{k+1}-X^k)\|_2^2+\langle\mathcal{A}(X^k-X^{k+1}), \mathcal{A}(X^{k+1})-\vb\rangle\\
&+\mu(\|X^k\|_1-\|X^{k+1}\|_1)+\lambda(\text{Tr}(X^k-X^{k+1})-\|X^k\|_F+\|X^{k+1}\|_F),
\end{split}
\]
and (\ref{eqn: inner mid}), we arrive at
\begin{equation}\label{eqn: f(x_k)-f(x_k+1)}
\begin{split}
F(X^k)-F(X^{k+1})=&\frac{1}{2}\|\mathcal{A}(X^{k+1}-X^k)\|_2^2+\mu(\|X^k\|_1-\langle X^k, G^{k+1}\rangle)+\langle X^k, \Lambda^{k+1}\rangle\\
&+\lambda(\|X^{k+1}\|_F-\langle X^{k+1}, Y^{k}\rangle)\\
\geq& 0.
\end{split}
\end{equation}
Here, the last inequality follows from $\|X^k\|_1-\langle X^k, G^{k+1}\rangle\geq 0,
\|X^{k+1}\|_F-\langle X^{k+1}, Y^{k}\rangle\geq 0$, and $\langle X^k,
\Lambda^{k+1}\rangle\geq 0$  since $\|G^{k+1}\|_\infty\leq 1$,
 $\|Y^{k}\|_F\leq 1$, and $\Lambda^{k+1}\succeq 0 $.
\end{proof}
We next show the convergence property of Algorithm \ref{alg1}.
 \begin{theo}
 Assume that $\{X^k\}_{k\geq 1}$ is a sequence generated by Algorithm \ref{alg1}. We have

 (1) $\{X^k\}_{k\geq 1}$ is a bounded sequence;

 (2)  $\lim_{k\rightarrow \infty}\|X^{k+1}-X^{k}\|_F=0$;

 (3) Assume that  $\widetilde{X}=\widetilde{X}_1+i\widetilde{X}_2$ is an
accumulation point of $\{X^k\}_{k\geq 1}$. Then $\widetilde{X}$ satisfies:
 \begin{enumerate}[(i)]
 \item Stationary condition:
\begin{equation}\label{eq:stapoint}
\lambda\left(\begin{bmatrix} \mathbf{I}\\\mathbf{0}\end{bmatrix}- \begin{bmatrix} \widetilde{X}_1 \\ \widetilde{X}_2 \end{bmatrix}\Big/\left\|\begin{bmatrix} \widetilde{X}_1 \\ \widetilde{X}_2 \end{bmatrix}\right\|_{F}\right)+\mu \begin{bmatrix} \widetilde{G}_1\\ \widetilde{G}_2\end{bmatrix} +\widetilde{\mathcal{A}}^*\left(\widetilde{\mathcal{A}}\left(\begin{bmatrix} \widetilde{X}_1 \\ \widetilde{X}_2 \end{bmatrix}\right)-\vb\right)-\begin{bmatrix} \widetilde{\Lambda}_1\\ \widetilde{\Lambda}_2 \end{bmatrix}=0;
\end{equation}
\item Complementary slackness condition:
\begin{equation}\label{eq:comslack}
\left\langle\begin{bmatrix} \widetilde{\Lambda}_1\\ \widetilde{\Lambda}_2 \end{bmatrix}, \begin{bmatrix} \widetilde{X}_1 \\ \widetilde{X}_2 \end{bmatrix}\right\rangle=0,
\end{equation}
for some  $\begin{bmatrix} {\widetilde{\Lambda}_1}\\ \widetilde{\Lambda}_2 \end{bmatrix}\in \mathbb{H}_+^{d\times d}$ and
   \begin{equation} \label{eqn: G_subgrad}
   \begin{bmatrix} \widetilde{G}_1\\ \widetilde{G}_2\end{bmatrix} \in \partial \left(\left\|\begin{bmatrix} \text{vec}{(\widetilde{X}_1)}\\ \text{vec}(\widetilde{X}_2) \end{bmatrix}\right\|_{1,2}\right),
   \end{equation}
   where
$\partial \left(\left\|\begin{bmatrix} \text{vec}{(\widetilde{X}_1)}\\
\text{vec}(\widetilde{X}_2)
\end{bmatrix}\right\|_{1,2}\right)$ is given in (\ref{eqn: l1_subgrad}).
   \end{enumerate}
 \end{theo}

\begin{proof}
(1) The definition of $F$ implies  $\mu\|X^{k+1}\|_1\leq F(X^{k+1})$ and hence
$\|X^{k+1}\|_1 \leq F(X^{k+1})/\mu \leq F(X^0)/\mu $ for $k\geq 1$. Here we use Lemma
\ref{eqn: Fk decreasing}, i.e., $\{F(X^{k})\}_{k\geq 1}$ is monotonically decreasing.
Hence, $\{X^{k}\}_{k\geq 1}$ is a bounded sequence.

 (2)
We first consider the case where $X^1=\mathbf{0}$. A simple calculation shows that $X^k=\mathbf{0}$ for $k\geq 2$ provided that $X^1=\mathbf{0}$, and we arrive at the conclusion immediately. So, we next just consider the case on  $X^1\neq \mathbf{0}$.
 Taking $k=0$ in (\ref{eqn: f(x_k)-f(x_k+1)}), we obtain that
 \[F(X^0)-F(X^1)=F(\mathbf{0})-F(X^1)=\frac{1}{2}\|\mathcal{A}(X^1)\|_2^2+\lambda\|X^1\|_F\geq \lambda\|X^1\|_F>0
 \]
 as $Y^k=\mathbf{0}$. It implies $F(X^k)\leq F(X^1)<F(\mathbf{0})$ for any $k\geq 2$. Hence, we obtain that
$X^k\neq 0$ for all $k\geq 1$. By  (\ref{eqn: f(x_k)-f(x_k+1)}), we obtain that
\[
F(X^k)-F(X^{k+1})\geq \frac{1}{2}\|\mathcal{A}(X^{k+1}-X^k)\|_2^2+\lambda(\|X^{k+1}\|_F-\langle X^{k+1},Y^k\rangle).
\]
Noting that $\{F(X^k)\}_{k\geq 1}$ is a convergent sequence and
$\|X^{k+1}\|_F-\langle X^{k+1},Y^k\rangle\geq 0$, we have
\begin{equation}\label{eqn: converg_mid1}
\lim_{k\rightarrow\infty }\|\mathcal{A}(X^k-X^{k+1})\|_2\,\,=\,\,0
\end{equation}
and
\begin{equation}\label{eqn: converg_mid2}
 \lim_{k\rightarrow\infty } (\|X^{k+1}\|_F-\langle X^{k+1},Y^{k}\rangle)\,\,=\,\,  \lim_{k\rightarrow\infty } \left(\|X^{k+1}\|_F-\left\langle X^{k+1},\frac{X^k}{\|X^k\|_F}\right\rangle\right)\ \ =\ \ 0.
\end{equation}
{ The following argument is similar with that in Proposition 3.1 (b) in \cite{YLX15}.
We put it here for completeness. Set $c_k:=\frac{\langle X^k,
X^{k+1}\rangle}{\|X^k\|_F^2}$ and $E^k:=X^{k+1}-c_k X^k$. It suffices to prove that
$E^k\rightarrow \mathbf{0}$ and $c_k\rightarrow 1$. According to  (\ref{eqn:
converg_mid2}) and boundness of $\{X^k\}_{k\geq 1}$,  we have
\[
\|E^k\|_F^2=\|X^{k+1}\|_F^2-\frac{\langle X^k, X^{k+1}\rangle^2}{\|X^k\|_F^2}=\left(\|X^{k+1}\|_F-\frac{\langle X^k, X^{k+1}\rangle}{\|X^k\|_F^2}\right)\left(\|X^{k+1}\|_F+\frac{\langle X^k, X^{k+1}\rangle}{\|X^k\|_F^2}\right)\rightarrow 0,
\]
Then we have
\[
0=\lim_{k\rightarrow\infty }\|\mathcal{A}(X^k-X^{k+1})\|_2=\lim_{k\rightarrow\infty }\|\mathcal{A}((c_k-1)X^k-E^k)\|_2=\lim_{k\rightarrow\infty}|c_k-1|\|\mathcal{A}(X^k)\|_2.
\]
If $\lim_{k\rightarrow\infty}c_k\neq 1$, then there exists a subsequence
$\{X^{k_j}\}$ such that $\|\mathcal{A}(X^{k_j})\|_2\rightarrow 0$. Therefore, we can
obtain that
\[
\lim_{k_j\rightarrow\infty} F(X^{k_j})\geq \lim_{k_j\rightarrow \infty}\frac{1}{2}\|\mathcal{A}(X^{k_j})-\mathbf{b}\|_2^2=\frac{1}{2}\|\mathbf{b}\|_2^2=F(X^0),
\]
which leads to a contradiction to the fact that
\[
F(X^{k_j})\leq F(X^1)<F(X^0).
\]
Thus we can get $c_{k}\rightarrow 1$, $E^k\rightarrow \mathbf{0}$, and thus
$X^{k+1}-X^{k}\rightarrow \mathbf{0}$, when $k\rightarrow \infty$. }

(3) Assume that $\{X^{k_j}\}_{j\geq 1}\subset \{X^k\}_{k\geq 1}$ is a subsequence
satisfying $\lim_{j\rightarrow \infty}
X^{k_j}=\widetilde{X}=\widetilde{X}_1+i\widetilde{X}_2\neq \mathbf{0}$. For
simplicity, we abuse the notation and denote $\{X^{k_j}\}$ as $\{X^{j}\}$. Replacing
$k$ by $j-1$ in (\ref{eqn: converge1}) and (\ref{eqn:converge2}), we have
\begin{equation}\label{eqn: converge11} \lambda\left(\begin{bmatrix}
\mathbf{I}\\\mathbf{0}\end{bmatrix}-
\begin{bmatrix} {Y_1^{j-1}} \\ {Y_2^{j-1}} \end{bmatrix}\right)+\mu \begin{bmatrix}
G_1^{j}\\ G_2^{j}\end{bmatrix}
+\widetilde{\mathcal{A}}^*\left(\widetilde{\mathcal{A}}\left(\begin{bmatrix}
{X_1^{j}} \\ {X_2^{j}} \end{bmatrix}\right)-\vb\right)=\begin{bmatrix}
{\Lambda_1^{j}}\\ {\Lambda_2^{{j}}} \end{bmatrix},
\end{equation}
and
\begin{equation}\label{eqn:converge12}
\left\langle\begin{bmatrix} {\Lambda_1^{j}}\\ {\Lambda_2^{j}} \end{bmatrix}, \begin{bmatrix} {X_1^{j}} \\ {X_2^{j}} \end{bmatrix}\right\rangle=0,
\end{equation}
for some $\Lambda^{j}=\Lambda_1^{j}+i \Lambda_2^{j}$ with  $\begin{bmatrix} {{\Lambda_1^{j}}}\\ {\Lambda_2^{j}} \end{bmatrix}\in \mathbb{H}_+^{d\times d}$,  and $G^{j}=G_1^{j}+i G_2^{j}$
   with
     \begin{equation}\label{eqn: G_j}
   \begin{bmatrix} {G_1^{j}}\\ {G_2^{j}}\end{bmatrix} \in \partial \left(\left\|\begin{bmatrix} \text{vec}{({X_1^{j}})}\\ \text{vec}({X_2^{j}}) \end{bmatrix}\right\|_{1,2}\right).
   \end{equation}
Note that   (\ref{eqn: converge11}) is equivalent to
\begin{equation}\label{eqn: converge 3}
\lambda\left(\begin{bmatrix} \mathbf{I}\\\mathbf{0}\end{bmatrix}- \begin{bmatrix} {Y_1^{j-1}} \\ {Y_2^{j-1}} \end{bmatrix}\right) +\widetilde{\mathcal{A}}^*\left(\widetilde{\mathcal{A}}\left(\begin{bmatrix} {X_1^{j}} \\ {X_2^{j}} \end{bmatrix}\right)-\vb\right)=\begin{bmatrix} {\Lambda_1^{j}}\\ {\Lambda_2^{{j}}} \end{bmatrix}-\mu \begin{bmatrix} G_1^{j}\\ G_2^{j}\end{bmatrix}.
\end{equation}
Noting that
\[
\lim_{j\rightarrow \infty}\begin{bmatrix} {Y_1^{j-1}} \\ {Y_2^{j-1}} \end{bmatrix} \,\,=\,\,
  \begin{bmatrix} \widetilde{X}_1 \\ \widetilde{X}_2 \end{bmatrix}\Big/\left\|\begin{bmatrix} \widetilde{X}_1 \\ \widetilde{X}_2 \end{bmatrix}\right\|_{F},
\]
we obtain that  the left hand side of (\ref{eqn: converge 3}) converges to
\begin{equation}\label{eqn: converge_LHS}
\lim_{j\rightarrow \infty} \lambda\left(\begin{bmatrix} \mathbf{I}\\\mathbf{0}\end{bmatrix}- \begin{bmatrix} {Y_1^{j-1}} \\ {Y_2^{j-1}} \end{bmatrix}\right) +\widetilde{\mathcal{A}}^*\left(\widetilde{\mathcal{A}}\left(\begin{bmatrix} {X_1^{j}} \\ {X_2^{j}} \end{bmatrix}\right)-\vb\right)=\lambda\left(\begin{bmatrix} \mathbf{I}\\\mathbf{0}\end{bmatrix}- \begin{bmatrix} \widetilde{X}_1 \\ \widetilde{X}_2 \end{bmatrix}\Big/\left\|\begin{bmatrix} \widetilde{X}_1 \\ \widetilde{X}_2 \end{bmatrix}\right\|_{F}\right)+\widetilde{\mathcal{A}}^*\left(\widetilde{\mathcal{A}}\left(\begin{bmatrix} \widetilde{X}_1 \\ \widetilde{X}_2 \end{bmatrix}\right)-\vb\right).
\end{equation}
For convenience, we set
\[
 P^j:=\begin{bmatrix} {\Lambda_1^{j}}\\ {\Lambda_2^{{j}}} \end{bmatrix}\quad\text{and}\quad Q^j:=-\mu \begin{bmatrix} G_1^{j}\\ G_2^{j}\end{bmatrix}.
\]
According to  (\ref{eqn: G_j}) and Proposition  \ref{prop: subgradient}, we have
$\|G_1^j\|_\infty\leq 1$ and $\|G_2^j\|_\infty\leq 1$. Combining  (\ref{eqn: converge
3}) and the boundedness of $\{X^j\}_{j\geq 1}$, we obtain that $\{P^j\}_{j\geq 1}$
and $\{Q^j\}_{j\geq 1}$ are also bounded sequences, which can belong to some compact
sets $ S\subset \mathbb{H}_{+}^{d\times d}$ and $T$, respectively.

We assume that  $\{{j_l}\}_{l\geq 1}$  is a subsequence of  $\{{j}\}_{j\geq 1}$ such
that $\lim_{l\rightarrow \infty}P_{j_l}=\widetilde{P}$ and $\lim_{l\rightarrow
\infty}Q_{j_l}=\widetilde{Q}$ for some $\widetilde{P}\in S$, $\widetilde{Q}\in T$.
More concretely, we have
\[
\lim_{l\rightarrow \infty}P^{j_l}=\lim_{l\rightarrow \infty}\begin{bmatrix} {\Lambda_1^{j_l}}\\ {\Lambda_2^{{j_l}}} \end{bmatrix}= \begin{bmatrix} {\widetilde{\Lambda}_1}\\ {\widetilde{\Lambda}_2} \end{bmatrix},\ \text{and}\ \lim_{l\rightarrow \infty}Q^{j_l}=\lim_{l\rightarrow \infty}-\mu\begin{bmatrix} G_1^{j_l}\\ G_2^{j_l}\end{bmatrix}=-\mu\begin{bmatrix} {\widetilde{G}_1}\\ {\widetilde{G}_2} \end{bmatrix}\]
for some
\[
\quad \begin{bmatrix} {\widetilde{\Lambda}_1}\\ {\widetilde{\Lambda}_2} \end{bmatrix}\in S\subset \mathbb{H}_{+}^{d\times d}.
\]
According to (\ref{eqn: converge 3}),  we have
\[
 \lambda\left(\begin{bmatrix} \mathbf{I}\\\mathbf{0}\end{bmatrix}-
\begin{bmatrix} \widetilde{X}_1 \\ \widetilde{X}_2
\end{bmatrix}\Big/\left\|\begin{bmatrix} \widetilde{X}_1 \\ \widetilde{X}_2
\end{bmatrix}\right\|_{F}\right)+\widetilde{\mathcal{A}}^*\left(\widetilde{\mathcal{A}}\left(\begin{bmatrix}
\widetilde{X}_1 \\ \widetilde{X}_2 \end{bmatrix}\right)-\vb\right)= \begin{bmatrix}
{\widetilde{\Lambda}_1}\\ {\widetilde{\Lambda}_2} \end{bmatrix}-\mu \begin{bmatrix}
{\widetilde{G}_1}\\ {\widetilde{G}_2} \end{bmatrix},
\]
which implies the stationary condition (\ref{eq:stapoint}). The complementary
slackness condition (\ref{eq:comslack}) is obtained by
\[
\left\langle\begin{bmatrix} {\widetilde{\Lambda}_1}\\
{\widetilde{\Lambda}_2} \end{bmatrix}, \begin{bmatrix} {\widetilde{X}_1} \\
{\widetilde{X}_2} \end{bmatrix}\right\rangle=\lim_{l\rightarrow \infty}\left\langle\begin{bmatrix} {\Lambda_1^{j_l}}\\ {\Lambda_2^{j_l}} \end{bmatrix},
\begin{bmatrix} {X_1^{j_l}} \\ {X_2^{j_l}} \end{bmatrix}\right\rangle\,\,=\,\,0.
\]
Here, we use (\ref{eqn:converge12}).

We remain to prove  (\ref{eqn: G_subgrad}). For sufficiently large $j_l$, we have
$\text{supp}(\widetilde{X})\subset \text{supp}(X^{j_l})$.  If $(i_1,i_2)\in
\text{supp}(\widetilde{X})$, then
\[
\lim_{l\rightarrow \infty}([G_1^{j_l}]_{i_1,i_2}, [{G_2^{j_l}}]_{i_1,i_2})= \frac{([\widetilde{X}_1]_{i_1,i_2}, [\widetilde{X}_2]_{i_1,i_2})}{\sqrt{[\widetilde{X}_1]_{i_1,i_2}^2+[\widetilde{X}_2]_{i_1,i_2}^2}}.
\]
  If $(i_1,i_2)\notin \text{supp}(\widetilde{X})$,  we have
\[
([G_1^{j_l}]_{i_1,i_2})^2+([{G_2^{j_l}}]_{i_1,i_2})^2\leq 1
\]
and hence
\[
([\widetilde{G}_1]_{i_1,i_2})^2+([{\widetilde{G}_2}]_{i_1,i_2})^2\leq 1.
\]
 Thus \[
\lim_{l\rightarrow \infty}\begin{bmatrix} G_1^{j_l}\\
G_2^{j_l}\end{bmatrix}=\begin{bmatrix} {\widetilde{G}_1}\\ {\widetilde{G}_2}
\end{bmatrix}\in \partial \left(\left\|\begin{bmatrix}
\text{vec}{(\widetilde{X}_1)}\\ \text{vec}(\widetilde{X}_2)
\end{bmatrix}\right\|_{1,2}\right),
\]
 which leads to (\ref{eqn: G_subgrad}).
\end{proof}
\section{Numerical experiments}\label{sec: numerical}

The purpose of numerical experiments is to compare the performance of  (\ref{eq:
unconstrained SDP}) with that of SPARTA   \cite{WZG16}, of SWF \cite{YWW17} and of
SPRSF \cite{PBA18}. We choose the parameters of  those algorithms as in
\cite{WZG16,YWW17,PBA18}. In this section,  we use the relative error
\[
\text{relative error}:=\frac{d_r(\vz,\vx_0)}{\|\vx_0\|_2},
\]
where $d_r(\vz,\vx_0)=\min \|\vz \pm \vx_0\|_2$ for the real case and
$d_r(\vz,\vx)=\min_{\theta\in [0,2\pi)}\|\exp(-i\theta)\vz -\vx\|_2$ for the complex
case. In our numerical experiments, we assume that the sampling vectors $\va_j,
j=1,\ldots,m$ are Gaussian random vector, i.e., $\va_j\sim
\mathcal{N}(0,\mathbf{I}_{d})$ for real case and  $\va_j\sim
\mathcal{N}(0,\frac{1}{2}\mathbf{I}_{d})+i\mathcal{N}(0,\frac{1}{2}\mathbf{I}_{d})$
for complex case.

 For each fixed $k$, the support of a $k$-sparse signal $\vx_0$
is drawn from the uniform distribution over the set of all subsets of $[1,m]\cap \Z$
of size $k$.  The non-zero entries of the real (resp. complex) $k$-sparse signal
$\vx_0$ have Gaussian distribution $\mathcal{N}(0,1)$ (resp.
$\mathcal{N}(0,1)+i\mathcal{N}(0,1)$). In order to reduce dimension effect, we
normalize $\vx_0$ into $\|\vx_0\|_2=1$. {{All experiments are carried out on Matlab
2017 with a 3.7 GHz Intel Core i7-8700K and 64 GB memory.}}

\begin{exam}
The aim of this numerical experiment is to test the success rate of Algorithm \ref{alg1}
against the measurement number $m$. In this example, we take $k=5$ and $d=50$. The
ratio between $m$ and $d$ is varied from $0.1$ to $4$, with stepsize $0.1$. We choose
$\mu=0.001$ and $\lambda=\frac{\mu k}{\sqrt{2}-1}$ in  Algorithm \ref{alg1}.
 We classify a recovery as a success if the relative error is less than $10^{-3}$. For
 each fixed $m$, we repeat the experiments for $40$ trails and and calculate the success
rate.


Figure \ref{fig: phase_transitionk=5} depicts the empirical probability of successful
recovery against the measurement number $m$. The numerical results show that
Algorithm \ref{alg1} outperform other algorithms.
\end{exam}
\begin{figure}[htbp]
	\centering
	\subfloat[]{{\includegraphics[width=0.5\textwidth]{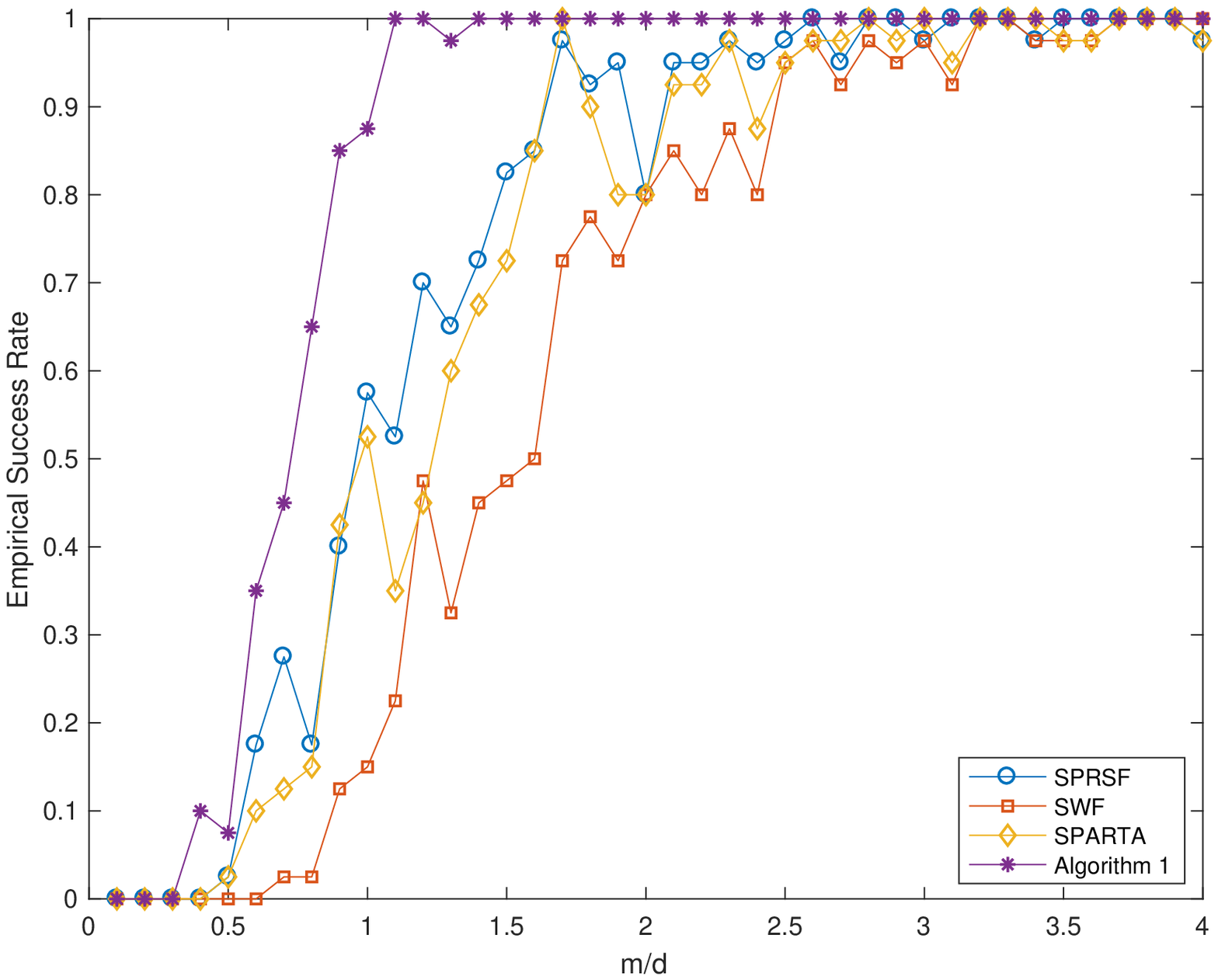}}}
	\subfloat[]{{\includegraphics[width=0.5\textwidth]{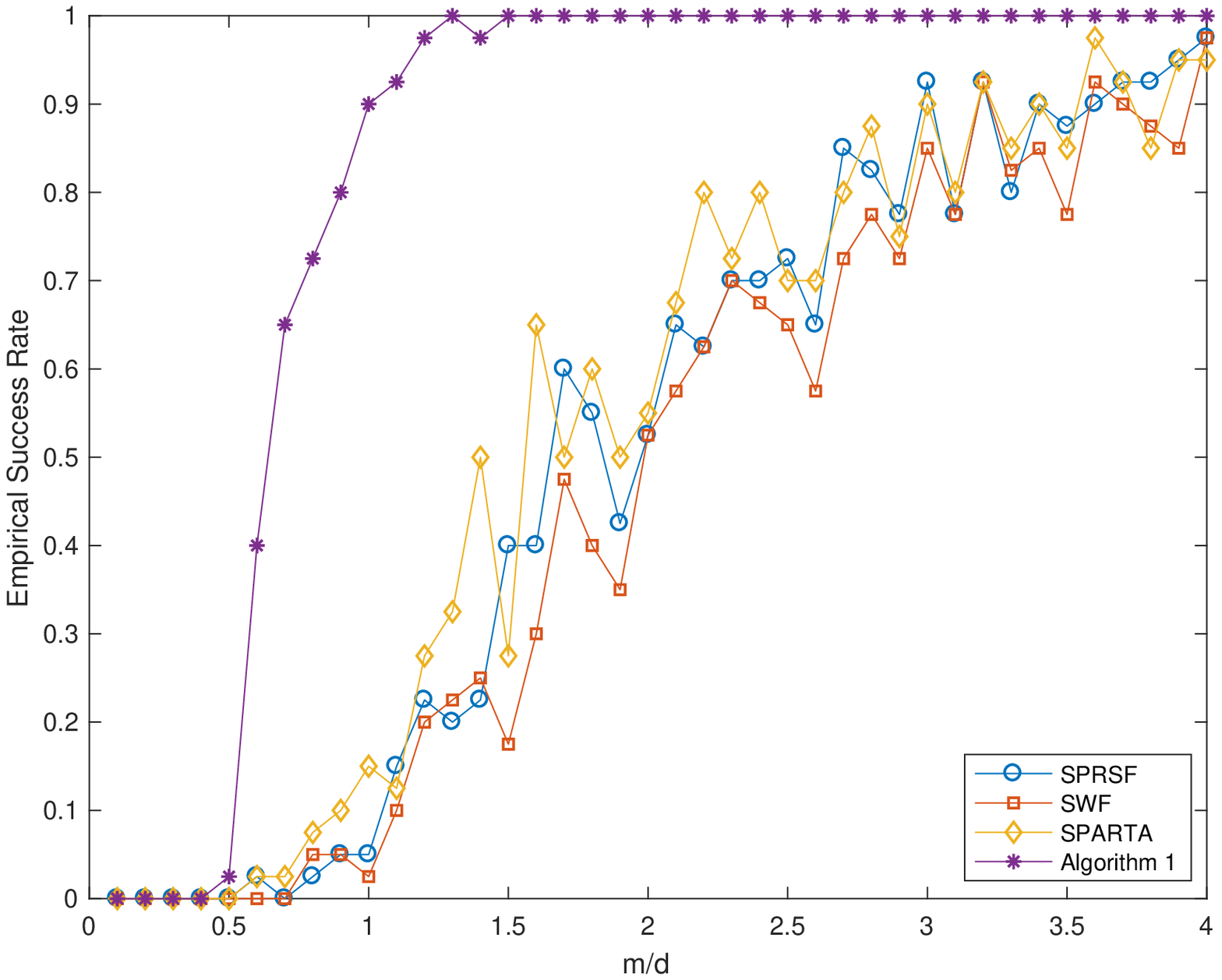}}}
	\caption{Comparison of different algorithms for fixed k = 5 with different $m/n$ ratio: (A) Noiseless real-valued Gaussian model; (B) Noiseless complex-valued Gaussian model.}
	\label{fig: phase_transitionk=5}
\end{figure}

\begin{exam}
In this example,  we test the success  rate of Algorithm \ref{alg1} against the sparsity level
$k$. We take $d=50$ and $m=2d$. The parameters in Algorithm \ref{alg1} are taken as $\mu=0.001$ and
$\lambda=\frac{\mu k}{\sqrt{2}-1}$.
 Figure \ref{fig: phase_transitionk}  depicts the numerical results.
   It shows that Algorithm \ref{alg1} is superior to the SPRSF, SWF and SPARTA  for
both real and complex cases. Furthermore, we can see that Algorithm \ref{alg1} can make good
performance even under large level of sparsity.
\end{exam}
\begin{figure}[htbp]
	\centering
	\subfloat[]{\label{phase_transition_gaussian}{\includegraphics[width=0.5\textwidth]{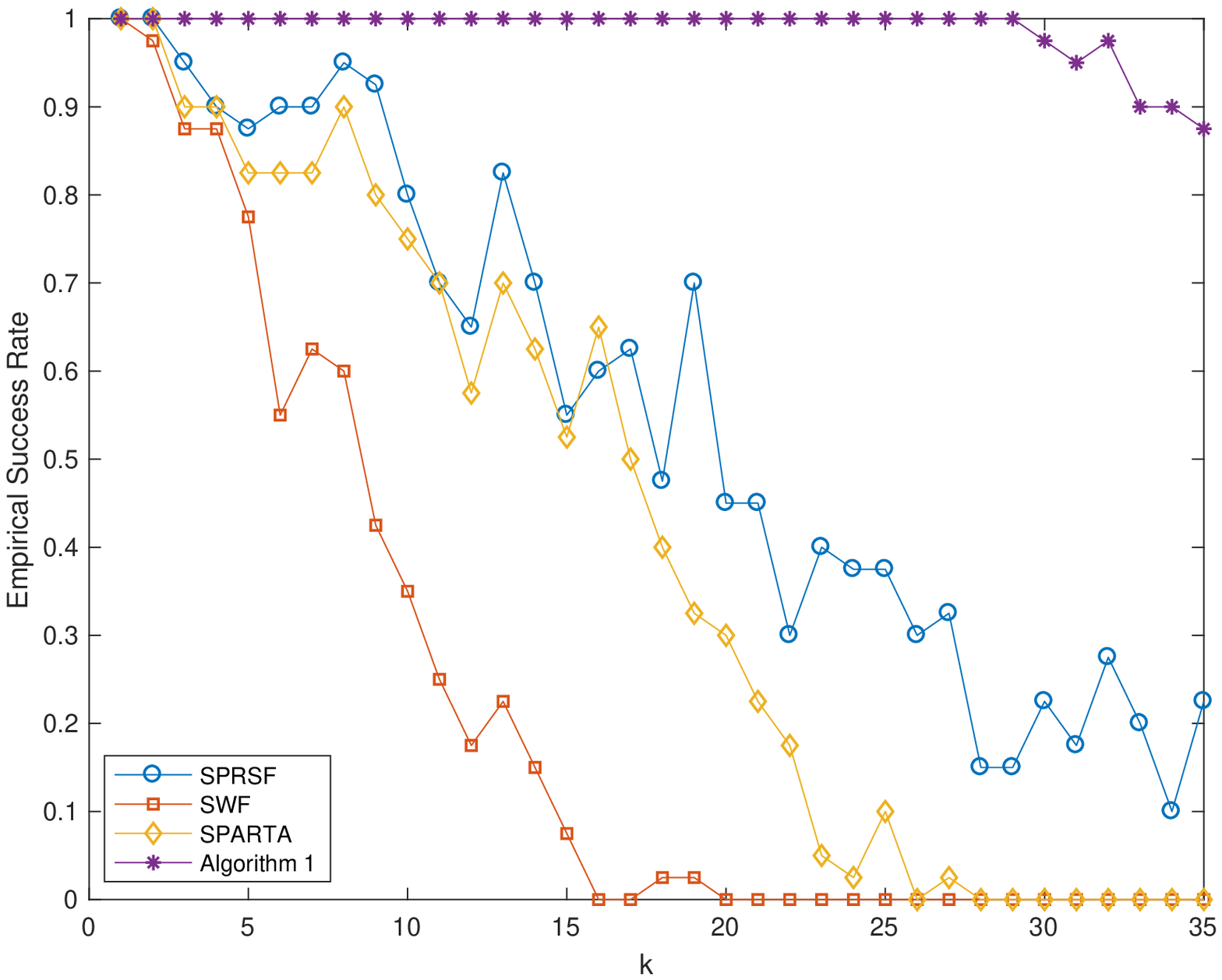}}}
	\subfloat[]{\label{phase_transition_uniform}{\includegraphics[width=0.5\textwidth]{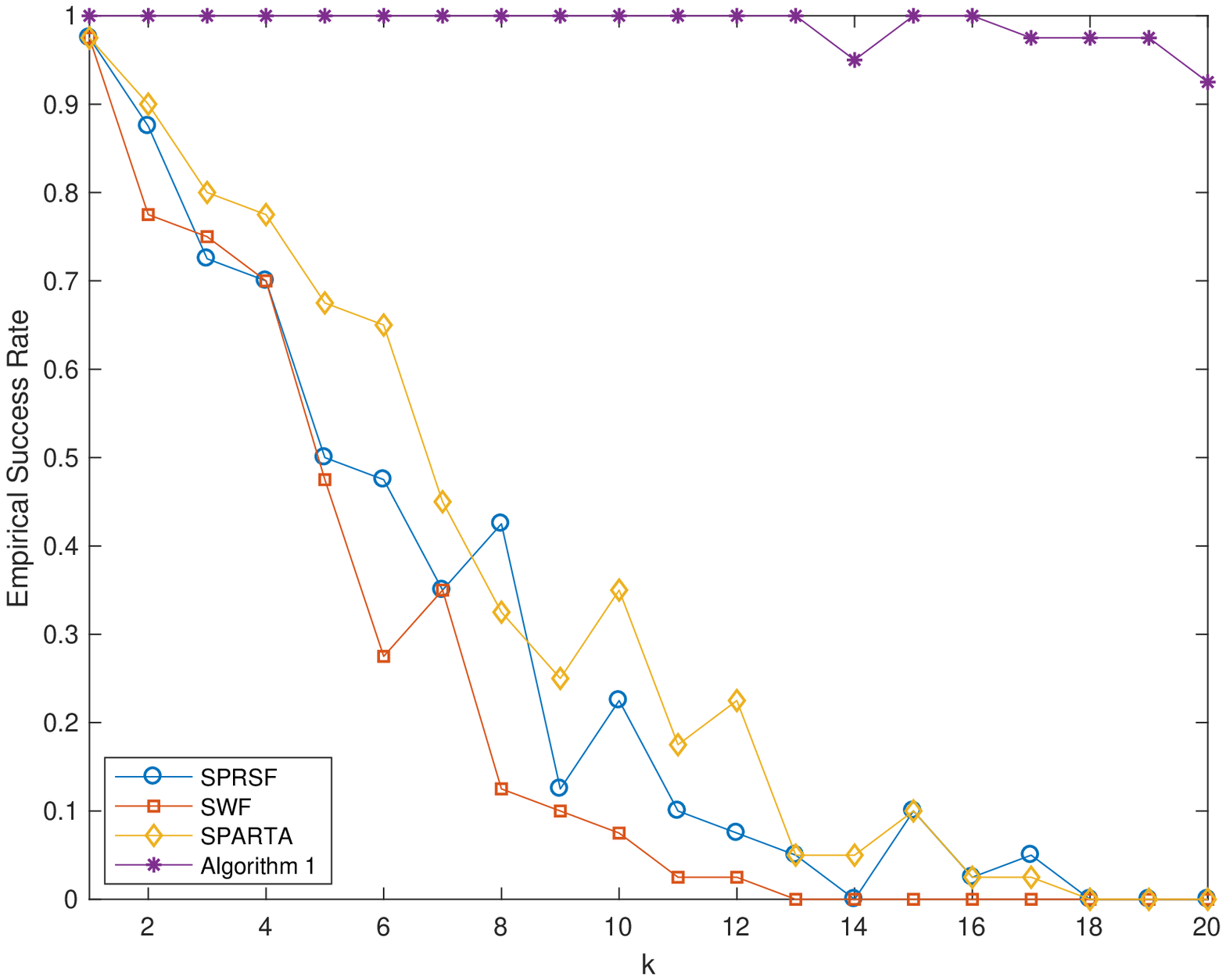}}}
	\caption{Comparison of different algorithms for different sparsity level $k$:
(A) Noiseless real-valued Gaussian model;  (B) Noiseless complex-valued Gaussian model.}
	\label{fig: phase_transitionk}
\end{figure}

\begin{exam} In this example, we test the robustness of
Algorithm \ref{alg1}. We take $d=50$, $m=2d$ and $k=5$ for both real and complex cases,
followed by adding white Gaussian noise by MATLAB function
\textbf{awgn($\mathcal{A}(\vx_0)$,snr)}, i.e., $b_j=\abs{\innerp{\va_j,\vx_0}}^2+w_j, j=1,\ldots,m$ with $\vw\sim \sqrt{\frac{\|\mathcal{A}(\vx_0)\|_2^2/m}{10^{snr/10}}}\mathcal{N}(0, \mathbf{I}_{m})$.
 Since other algorithms do not make $100\%$ recovery
under this setting, we only show the robustness performance on Algorithm \ref{alg1}.
The SNR value varies from 10dB to 50dB, with step-size 5dB. The SNR in each noise
level is averaged over 20 independent trials. According to Theorem \ref{th:rank1mod},
we choose $\mu=\max\{0.5\|\vw\|_2,0.001\}$ and $\lambda=\frac{\mu k}{\sqrt{2}-1}$. We
compute the signal-to noise ratio of reconstruction in dB as
$-20\log_{10}(\text{relative error})$.  In Figure \ref{fig: robustness}, it shows
that Algorithm \ref{alg1}  yields robust recovery with respect to different noise
level. In addition, the recovery error is a bitter larger for complex case.
\end{exam}
\begin{figure}[htbp]
	\centering
{\includegraphics[width=0.5\textwidth]{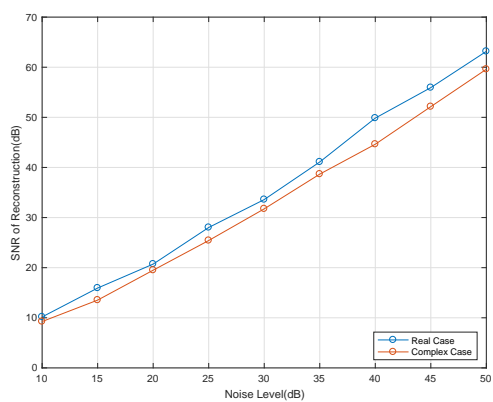}}
	\caption{SNR of signal recovery v.s. noise level in measurements when $k=5$: $x$-axis is the noise level varying from $10$db to $50$db, $y$-axis is the reconstruction error in db as $-20\log_{10}(\text{relative error})$.}
	\label{fig: robustness}
\end{figure}

\end{document}